\documentclass[a4paper,12pt]{amsart}
\usepackage{latexsym,amscd,amssymb,url}
\usepackage[all,cmtip]{xy}  
\usepackage{graphicx}
\usepackage{xcolor}
\usepackage{enumerate} 
\definecolor{dblue}{rgb}{0,0,.6}
\usepackage[colorlinks=true, linkcolor=dblue, citecolor=dblue, filecolor = dblue, menucolor = dblue, urlcolor = dblue]{hyperref}

\newtheorem{theorem}{Theorem}[section]

\theoremstyle{plain}

\newtheorem{corollary}[theorem]{Corollary}

\newtheorem{definition}[theorem]{Definition}

\newtheorem{lemma}[theorem]{Lemma}

\newtheorem{proposition}[theorem]{Proposition}
\newtheorem{remark}[theorem]{Remark}

\newcommand{\del}{\partial}
\newcommand{\Z}{\mathbb Z}
\newcommand{\Q}{\mathbb Q}

\newcommand{\C}{\mathbb C}

\newcommand{\CP}{\mathbb P}

\newcommand{\Hom}{\operatorname{Hom}}

\newcommand{\Spec}{\operatorname{Spec}}

\newcommand{\CH}{\operatorname{CH}}

\newcommand{\sing}{\operatorname{sing}} 
\newcommand{\red}{\operatorname{red}}

\newcommand{\Frac}{\operatorname{Frac}}
\newcommand{\Val}{\operatorname{Val}}
\newcommand{\F}{\mathbb F}

  \newcommand{\Tor}{\operatorname{Tor}}
  \newcommand{\Pf}{\operatorname{Pf}}
  \newcommand{\A}{\mathbb{A}}

\newcommand{\dashedlongrightarrow}{\xymatrix@1@=15pt{\ar@{-->}[r]&}}
\renewcommand{\longrightarrow}{\xymatrix@1@=15pt{\ar[r]&}}
\renewcommand{\mapsto}{\xymatrix@1@=15pt{\ar@{|->}[r]&}}
\renewcommand{\twoheadrightarrow}{\xymatrix@1@=15pt{\ar@{->>}[r]&}}
\newcommand{\hooklongrightarrow}{\xymatrix@1@=15pt{\ar@{^(->}[r]&}}
\newcommand{\congpf}{\xymatrix@1@=15pt{\ar[r]^-\sim&}}
\renewcommand{\cong}{\simeq}

\begin{document}    

\title[Torsion orders of Fano hypersurfaces]{Torsion orders of Fano hypersurfaces}

\author{Stefan Schreieder} 
\address{Institute of Algebraic Geometry, Leibniz University Hannover, Welfengarten 1, 30167 Hannover , Germany.}
\email{schreieder@math.uni-hannover.de}

\date{June 9, 2020} 
%


\begin{abstract}    We find new lower bounds on the torsion orders of very general Fano hypersurfaces over (uncountable) fields of arbitrary characteristic. Our results imply that unirational parametrizations of most Fano hypersurfaces need to have very large degree. Our results also hold in characteristic two, where they solve the rationality problem for hypersurfaces under a logarithmic degree bound, thereby extending a previous result of the author from characteristic different from two to arbitrary characteristic. 
\end{abstract}

\maketitle

\section{Introduction}

Let $X$ be a projective variety over a field $k$.
The torsion order $\Tor(X)$ of $X$ is  the smallest positive integer $e$, such that $e$ times the diagonal of $X$ admits a decomposition in the Chow group of $X\times X$, that is,
$$
e\Delta_X=[z\times X]+B\in \CH_{\dim X}(X\times X),
$$
where $z\in \CH_0(X)$ is a zero-cycle of degree $e$ and $B$ is a cycle on $X\times X$ that does not dominate the second factor. 
If no such decomposition exists, we put $\Tor(X)=\infty$. 
If $X$ is smooth and $k$ is algebraically closed, then $\Tor(X)$ is the smallest positive integer $e$ such that for any field extension $L$ of $k$, the kernel of the degree map $\CH_0(X_L)\to \Z$ is $e$-torsion, and $\Tor(X)=\infty$ if no such integer exists.
This notion goes back to Bloch \cite{bloch} (using an idea of Colliot-Th\'el\`ene) and Bloch--Srinivas \cite{BS}, and has for instance been studied in \cite{ACTP} and \cite{voisin}, and in the above form by Chatzistamatiou--Levine \cite{CL} and Kahn \cite{kahn}.

The torsion order is a stable birational invariant of smooth projective varieties; it is finite if $X$ is rationally connected and it is $1$ if $X$ is stably rational.
Moreover, if $f:Y\to X$ is a generically finite morphism, then $\Tor(X)$ divides $\deg(f)\cdot \Tor(Y)$.
In particular, the degree of any unirational parametrization of $X$ is divisible by $\Tor(X)$.

The torsion order is a powerful invariant of rationally connected varieties, which we would like to compute for interesting classes of varieties. 
In particular, it is desirable to do so 
for smooth hypersurfaces $X_d\subset \CP^{N+1}_k$ of degree $d\leq N+1$.
By a result of Roitman \cite{roitman} and Chatzistamatiou--Levine \cite[Proposition 5.2]{CL},  we have
\begin{align} \label{eq:Tor(Xd)|d!}
\Tor(X_d)\mid d! \, .
\end{align}
This yields an upper bound which holds over any field $k$. 

Finding lower bounds for $\Tor(X_d)$ over algebraically closed fields is in general a difficult problem.
By a result of Chatzistamatiou--Levine \cite[Theorem 8.2]{CL}, building on earlier work of Totaro \cite{totaro} and Koll\'ar \cite{kollar}, the torsion order of a very general complex hypersurface $X_d\subset \CP^{N+1}_\C$ of degree $
d\geq p^j\lceil \frac{N+2}{p^j+1}\rceil
$
 is divisible by $p^j$, where $p$ denotes a prime number, and where $p$ is odd or $N$ is even.
This yields 
non-trivial lower bounds roughly  in degrees $d> \frac{2}{3}N$.  
In \cite{Sch-JAMS}, the author dealt with lower degrees by showing that  
the torsion order of a very general hypersurface $X_d\subset \CP^{N+1}_\C$ of degree $d\geq \log_2N+2$ and dimension $N\geq 3$ is divisible by $2$.
This paper generalizes that result as follows.

\begin{theorem}\label{thm:torsion-order}
Let $k$ be an uncountable field.
Then the torsion order of a very general Fano hypersurface $X_d\subset \CP^{N+1}_{k}$ of degree $d\geq 4$ is divisible by  
every integer $m\leq d-\log_2N$ that is invertible in $k$. 
\end{theorem}

The first new case concerns very general quintic fourfolds $X_5\subset \CP^5_k$ over (algebraically closed) 
fields of characteristic different from $3$, for which we get $3\mid \Tor(X_5)$.
If $\operatorname{char} k=0$, then $\Tor(X_5)$ is also divisible by $2$ and $5$ (see \cite{totaro} and \cite[Theorem 8.2]{CL}) and so our result determines all prime factors of $\Tor(X_5)$ by (\ref{eq:Tor(Xd)|d!}).

The strength of Theorem \ref{thm:torsion-order} lies in its asymptotic behaviour for large $N$.
To illustrate this, let $X_{100}\subset \CP^{100}_\C$ be a very general complex hypersurface of degree $100$.
Then $\Tor(X_{100})$ is divisible by 
$$
2^5 \cdot 3^3\cdot 5^2\cdot 7\cdot \prod_{\substack{p\leq 89\\  p\ \text{prime}}} p =718\, 766\, 754\, 945\, 489\, 455\, 304\, 472\, 257\, 065\, 075\, 294\, 400,
$$
while it was previously only known to be divisible by $2^3\cdot 3^2\cdot 5^2\cdot 7\cdot 11=138\, 600$. 

Even though smooth hypersurfaces $X_d\subset \CP_\C^{N+1}$ of degree $d$ with $d!\leq \log_2(N+1)$ are known to be unirational \cite{HMP,BR}, very general Fano hypersurfaces of large degree are conjecturally not unirational.
While this paper does not solve this problem, it does show that for most Fano hypersurfaces, unirational parametrizations  
need to have enormously large degree, strengthening previous bounds on this problem: In \cite[Theorem 4.3]{kollar}, Koll\'ar gave lower bounds on the degree of a uniruled parametrization of high-degree Fano hypersurfaces, and, relying on \cite{totaro}, Chatzistamatiou--Levine produced slightly better bounds for unirational parametrizations in \cite[Theorem 8.2]{CL}.

In \cite{Sch-JAMS} it was shown that very general hypersurfaces of dimension $N\geq 3$ and degree $d\geq \log_2N+2$ are stably irrational over any uncountable field of characteristic different from two.
This improved \cite{kollar,totaro} in characteristic $\neq 2$,  but the Koll\'ar--Totaro bound $d\geq 2\lceil\frac{N+2}{3}\rceil$ remained the best known result in characteristic two.

Applying Theorem \ref{thm:torsion-order} to $m=3$,  
this paper solves the rationality problem for hypersurfaces in characteristic two under a logarithmic degree bound.

\begin{corollary}\label{cor:torsion-order-char2}
Let $k$ be an uncountable field of characteristic two.
Then a very general hypersurface $X\subset \CP^{N+1}_{k}$ of degree $d\geq \log_2N+3$ is stably irrational.
\end{corollary}
 
The method of this paper is flexible and applies to other types of varieties as well.
To illustrate this, we include here the example of cyclic covers of projective space.

\begin{theorem} \label{thm:cyclic-cover:intro}
Let $k$ be an uncountable field and let $m\geq 2$ be an integer that is invertible in $k$.
Then the torsion order of a cyclic $m:1$ cover $X\to \CP^{N}_{k}$ branched along a very general hypersurface of degree $d\geq m (\lceil \frac{\lceil \log_2N \rceil+1}{m}\rceil +2)$ (with $m\mid d$) 
is divisible by $m$.
\end{theorem}

In particular, under the above degree bound, very general cyclic $m:1$ covers are stably irrational.
Even for $k=\C$, 
this extends previous results on this problem substantially, see \cite{kollar2,CTP2,oka}. 
For $m=2$, a similar result has previously been proven in \cite[Theorem 9.1]{Sch-JAMS}.
 
The above results are proven via a version of the degeneration technique that the author developed in \cite{Sch-duke,Sch-JAMS} and which improved the method of Voisin \cite{voisin} and Colliot-Th\'el\`ene--Pirutka \cite{CTP}.
An essential ingredient in this approach is the construction of varieties  
 that have nontrivial unramified cohomology. 

Constructing rationally connected varieties with nontrivial unramified cohomology is a subtle problem.
In degree two, the first examples are due to Saltman \cite{saltman}.
Building on \cite{artin-mumford}, the first examples in degree three and with $\Z/2$-coefficients have been constructed in \cite{CTO}.
This has later been generalized to arbitrary degrees and $\Z/\ell$-coefficients for any prime $\ell$ in \cite{peyre,asok}.
Starting with \cite{CTO}, all these constructions rely on norm varieties attached to symbols in Milnor K-theory mod $\ell$. 
Norm varieties attached to symbols of length two are Brauer--Severi varieties. 
For $\ell\neq 2$, such varieties have large degree, compared to their dimensions, which hints that they are not useful for our purposes. 
Moreover, for symbols of length at least three, norm varieties for $\ell\neq 2$ are very intricate objects, whose construction, due to Rost, relies inductively on the Bloch--Kato conjecture in lower degrees, see \cite{SJ}.
The situation is special for $\ell=2$, where norm varieties are Pfister quadrics, 
which are much simpler objects. 
Pfister quadrics are used in \cite{Sch-JAMS}, which explains the restriction to the prime $2$.

This paper introduces 
 for any integer $m$ large classes of  hypersurfaces with unramified $\Z/m$-cohomology, see Theorem \ref{thm:unramified-coho} below.
 As in \cite{Sch-JAMS}, an important ingredient is a quite flexible degeneration argument which allows to prove nontriviality of certain classes without any deep result from K-theory, see item (\ref{item:nonzero}) in Theorem \ref{thm:unramified-coho}. 
Besides the ideas from \cite{Sch-JAMS}, the main new ingredient of this paper is the definition and usage of universal relations in Milnor K-theory, see Definition \ref{def:universal-relation} below.
Our approach is elementary, works for any positive integer $m$ (not necessarily prime) and does not rely on norm varieties, nor on Voevodsky's proof of the Bloch--Kato conjectures.
%

Asok's examples \cite{asok} with nontrivial unramified $\Z/\ell$-cohomology in degree $n$ have dimension $N\gg \ell^n$, which grows rapidly with $n$.
For any given prime $\ell$ and integer $N\geq 3$, this led Asok \cite[Question 4.5]{asok} to ask  for  general restrictions on the possible degrees in which rationally connected complex varieties of dimension $N$ can have nontrivial unramified $\Z/\ell$-cohomology. 
This is a quite subtle problem already for rationally connected threefolds, where by a result of Voisin \cite{voisin-integral-hodge} and Colliot-Th\'el\`ene--Voisin \cite[Th\'eor\`eme 1.2]{CTV}, it boils down to understanding the possible Brauer groups,
 see \cite[Remarks 4.7, 4.8 and 4.10]{asok}.  

As a consequence of our proof of Theorem \ref{thm:torsion-order}, we obtain the following uniform result in arbitrary dimension; the case $m=2$ is due to \cite[Theorem 1.5]{Sch-JAMS}.

\begin{theorem} \label{thm:asoks-question}
Let $m,n\geq 2$ and $N\geq 3$ be integers with $ \log_2(m+1)\leq n\leq N+1-m$.
Then there is a rationally connected smooth complex projective variety $X$ of dimension $N$ such that the $n$-th unramified cohomology  $H^n_{nr}(\C(X)/\C,\Z/m)$ of $X$ contains an element of order $m$.
\end{theorem}

It is natural to wonder whether the upper bound in the above theorem is sharp.
For instance, is 
the unramified $\Z/2$-cohomology of 
a rationally connected smooth complex projective variety $X$ 
 trivial in degree $n=\dim X$? 

\begin{remark}
The main results of this paper are formulated over uncountable fields.
However, our proofs show that one can also write down examples over small fields, see e.g.\ Section \ref{sec:examples} below for some explicit examples over $\Q(t)$ and $\F_p(t,s)$.
 Moreover, the varieties in Theorem \ref{thm:asoks-question} may be chosen to be defined over $\Q$, see Remark \ref{rem:thm:asok:Q} below. 
\end{remark}

\section{Preliminaries}

\subsection{Conventions}
A variety is an integral separated scheme of finite type over a field.
For a scheme $X$, we denote its codimension one points by $X^{(1)}$.
A property holds for a very general point of a scheme if it holds at all closed points inside some countable intersection of open dense subsets.   

\subsection{Degenerations} \label{subsec:degeneration}
Let $R$ be a discrete valuation ring with fraction field $K$ and algebraically closed residue field $k$.
Let $\mathcal X\to \Spec R$ be a proper flat morphism with generic fibre $X$ and special fibre $Y$.
Then we say that $X$ degenerates to $Y$.
We also say that the base change of $X$ to any larger field degenerates (or specializes) to $Y$.
For instance, if $\mathcal X\to B$  is a proper flat morphism of varieties over an algebraically closed uncountable field, then the fibre $X_t$ over a very general point $t\in B$ degenerates to the fibre $X_0$ for any closed point $0\in B$ in the above sense, see e.g.\ \cite[\S 2.2]{Sch-duke}.
In particular, a very general hypersurface $X\subset \CP^{n+1}_k$ over an algebraically closed uncountable field $k$ specializes to any given hypersurface of the same dimension and degree over $k$.

\subsection{Alterations} \label{subsec:alteration}
Let $Y$ be a variety over an algebraically closed field $k$.
An alteration of $Y$ is a proper generically finite surjective morphism $\tau:Y'\to Y$, where $Y'$ is a non-singular variety over $k$.
De Jong \cite{deJ} proved that alterations always exist. 
Later, Gabber showed that one can additionally require that $\deg(\tau)$ is prime to any given prime $\ell\neq \operatorname{char}(k)$.
Temkin \cite[Theorem 1.2.5]{temkin} generalized this further, ensuring that $\deg (\tau)$ is a power of the characteristic of $k$ (or one if $\operatorname{char}(k)=0$).

\subsection{Milnor K-theory}
Let $L$ be a field.
Recall that Milnor K-theory $K_n^M(L)$ of $L$ in degree $n\geq 2$ is defined as the quotient of $(L^\ast)^{\otimes n}$, where $L^\ast$ denotes the multiplicative group of units in $L$, by the subgroup generated by tensors of the form $a_1\otimes\dots \otimes a_n$ with $a_i+a_{i+1}=1$ for some $1\leq i\leq n-1$.
Moreover, $K_0^M(L)=\Z$ and $K_1^M(L)=L^\ast$.
The image of a tensor $a_1\otimes\dots \otimes a_n$ in $K_n^M(L)$ is denoted by $(a_1,\dots ,a_n)$.
The direct sum $K_\ast^M(L):=\bigoplus _{n\geq 0} K_n^M(L)$ has a natural product structure, induced by the tensor product.

For an integer $m\geq 2$, the defining relation for Milnor K-theory implies the following basic relation in Milnor K-theory mod $m$, see  \cite[Lemma 1.3]{milnor}.

\begin{lemma} \label{lem:(b1,...,bn)=0}
Let $L$ be a field and let $b_1,\dots ,b_n\in L^\ast$ such that $\sum b_i=c^m$ for some $c\in L$.
Then
$
(b_1,\dots ,b_n)=0\in K_n^M(L)/m .
$
\end{lemma}

Let $A$ be a ring and let $A^*$ be the multiplicative group of units in $A$.
We define $K^m_n(A)$ as the quotient of $(A^*)^{\otimes n}$ by the subgroup generated by $a_1\otimes \dots \otimes a_n$ with $a_i+a_{i+1}=1$ for some $1\leq i \leq n-1$.
If $A$ is a field, then this definition coincides with the one above.
If $A\to B$ is a homomorphism of rings, then we obtain an induced homomorphism $K^M_n(A)\to K^M_n(B)$. 
In particular, if $A$ is an integral domain with fraction field $L$, then there is a natural map $\psi_L:K^M_n(A)\to K^M_n(L)$, and if $A$ is local with residue field $\kappa$, then there is a natural map $\psi_\kappa:K^M_n(A)\to K^M_n(\kappa)$.
The following result is known as specialization property in Milnor K-theory, c.f.\ \cite[Definition 2.1.4.c]{CT}.

\begin{lemma} \label{lem:spec-prop}
Let $A$ be a regular local ring with fraction field $L$ and residue field $\kappa$.
Then for any integer $m\geq 1$, we have
$$
\ker \left(\psi_L: K^M_n(A)/m \longrightarrow K^M_n(L)/m \right) \subseteq \ker \left( \psi_\kappa: K^M_n(A)/m \longrightarrow K^M_n(\kappa)/m \right) .
$$
\end{lemma}
\begin{proof}
By \cite[Lemma 2.1.5(a)]{CT}, it suffices to prove the lemma in the case where $A$ is a complete discrete valuation ring.
In this case, let $\pi\in A$ be a uniformizer.
This induces a residue homomorphism $\del_\pi:K^M_{n+1}(L)/m\to K^M_{n}(\kappa)/m$ (which in fact depends only on $A$ and not on the choice of the uniformizer $\pi$), such that for $\alpha \in K^M(A)/m$, 
$$
\psi_\kappa(\alpha)=\del_\pi((\pi)\otimes \psi_L(\alpha)) ,
$$
where $(\pi)\in K^M_1(L)/m=L^\ast/(L^\ast)^m$, see \cite[Lemma 2.1]{milnor}.
This immediately shows $\ker(\psi_L)\subset \ker(\psi_\kappa)$, as we want.
\end{proof}

\subsection{Galois cohomology and unramified cohomology}
Let $m\geq 2$ be a positive integer and let $L$ be a field in which $m$ is invertible.
We denote by $\mu_m\subset L$ the \'etale sheaf of  $m$-th roots of unity.
For an integer $j\geq 1$, we consider the twists $ \mu_m^{\otimes j}:=\mu_m\otimes\dots \otimes \mu_m $ (j-times) and define $\mu^{\otimes 0}:=\Z/m$ and $  \mu_m^{\otimes j}:=\Hom(\mu_m^{\otimes - j},\Z/m)$ for $j<0$.
For a field $L$ in which $m$ is invertible and which contains a primitive $m$-th root of unity, there is, for any integer $j$, an isomorphism $ \mu_m^{\otimes j}(\Spec L)\cong \Z/m$.

Let $L$ be a field and let $m$ be an integer that is invertible in $L$.
For any integer $j$, we denote by $H^i(L,\mu_m^{\otimes j})$ the Galois cohomology of the absolute Galois group of $L$ with the natural action on $\mu_m^{\otimes j}(\Spec L^{sep})$, where $L^{sep}$ denotes a separable closure of $L$.
Kummer theory induces a canonical isomorphism $H^1(L,\mu_m)\cong L^\ast/(L^\ast)^m=K^M_1(L)/m$.
By \cite[Theorem 3.1]{tate}, this induces via cup products a morphism of graded rings 
\begin{align} \label{eq:bass-tate}
K^M_\ast(L)/m\longrightarrow H^\ast(L,\mu_m^{\ast}):=\bigoplus_{i\geq 0} H^i(L,\mu_m^{\otimes i}) .
\end{align}  
(In fact, this map is an isomorphism by the Bloch--Kato conjecture, proven by Voevodsky, but we will not use this fact in this paper.) 
By slight abuse of notation, we denote the image of a class $(a_1,\dots ,a_n)\in K^M_n(L)/m$ in $H^n(L,\mu_m^{\otimes n})$ by the same symbol.

Let $A$ be a ring with $\frac{1}{m}\in A$ and let $L$ be a field.
Let $A\to L$ be a homomorphism of rings.
%
Since $H^i(L,\mu_m^{\otimes j})$ coincides with the \'etale cohomology of $\Spec L$ with values in the \'etale sheaf $\mu_m^{\otimes j}$, there is a natural pullback map $H^i_{\text{\'et}}(\Spec A,\mu_m^{\otimes j})\to H^i(L,\mu_m^{\otimes j})$.
This applies in particular to the case where $A$ with $\frac{1}{m}\in A$ is a local ring with residue field $L$, or an integral domain with fraction field $L$.
By \cite[Lemma 2.1.5(b) and \S 3.6]{CT}, the following specialization property (c.f.\ \cite[Definition 2.1.4.c]{CT}),  analogues to Lemma \ref{lem:spec-prop}, holds in \'etale cohomology.

\begin{lemma} \label{lem:spec-prop-etale}
Let $A$ be a regular local ring with fraction field $L$ and residue field $\kappa$.
Then for any integer $m\geq 1$ that is invertible in $\kappa$, we have
$$
\ker \left( H^i_{\text{\'et}}(\Spec A,\mu_m^{\otimes j}) \to H^i (L, \mu_m^{\otimes j}) \right)  \subseteq \ker \left( H^i_{\text{\'et}}(\Spec A, \mu_m^{\otimes j}) \to H^i(\kappa,\mu_m^{\otimes j} ) \right)  .
$$
\end{lemma}

For any discrete valuation $\nu$ on a field $L$, such that $m$ is invertible in the residue field $\kappa(\nu)$, there is a residue map
$$
\del_\nu:H^i(L,\mu_m^{\otimes j})\longrightarrow H^{i-1}(\kappa(\nu),\mu_m^{\otimes (j-1)}) ,
$$
which (for $i=j$) is compatible with the aforementioned residue map in Milnor K-theory. 
This map is surjective and its kernel is described as follows, see e.g.\ \cite[(3.10)]{CT}.

\begin{theorem} \label{thm:ses}
Let $\nu$ be a discrete valuation on a field $L$ such that $m$ is invertible in the residue field $\kappa(\nu)$.
Let $\mathcal O_{\nu}\subset L$ be the associated discrete valuation ring. 
Then the natural sequence
\begin{align} \label{eq:ses}
0\longrightarrow H^i_{\text{\'et}}(\Spec \mathcal O_{\nu},\mu_m^{\otimes j})\longrightarrow H^i(L,\mu_m^{\otimes j})\stackrel{\del_\nu} \longrightarrow H^{i-1}(\kappa(\nu),\mu_m^{\otimes (j-1)})\longrightarrow 0 .
\end{align} 
is exact. 
\end{theorem}
 
Assume now that $L=k(X)$ is the function field of a $k$-variety $X$ and let $\Val(L/k)$ be the set of all valuations on $L$ that are induced by a prime divisor on some normal birational model of $X$.
The unramified $\mu_m^{\otimes j}$-cohomology of $X$ in degree $i$ is defined as
$$
H^i_{nr}(k(X)/k,\mu_m^{\otimes j}):=\{\alpha \in H^i(k(X),\mu_m^{\otimes j})\mid \del_\nu \alpha =0\ \ \forall \nu\in \Val(k(X)/k)\} .
$$ 
This subgroup of $H^i(k(X),\mu_m^{\otimes j})$ is a stable birational invariant of $X$, see \cite{CTO}.
If $k$ contains a primitive $m$-th root of unity, then $\mu_m^{\otimes j}\cong \mu_m^{\otimes 0}= \Z/m$ for all $j$ and so $H^i_{nr}(k(X)/k,\mu_m^{\otimes j})\cong H^i_{nr}(k(X)/k,\Z/m)$.

Let $\gamma\in H^i_{nr}(k(X)/k,\mu_m^{\otimes j})$ be unramified and let $E\subset X$ be a subvariety whose generic point $x$ lies in the smooth locus of $X$.
Then $\gamma$ lifts uniquely to a class in the cohomology of $\Spec \mathcal O_{X,x}$ (see  e.g.\ \cite[Theorem 3.8.2]{CT}) and so it can be restricted to the closed point to give a class in $H^i(\kappa(x),\mu_m^{\otimes j})=H^i(k(E),\mu_m^{\otimes j})$ that we denote by $\gamma|_x$ or $\gamma|_E$.

\section{Universal relations in Milnor K-theory modulo $m$}
\label{sec:universal-rel}

Fix a base field $k$ and a natural number $m\geq 2$.
For integers $n,s\geq 1$, let 
$$
R_{n,s}:=k[x_1,x_2,\dots ,x_n,y_1,\dots ,y_{s}]
$$ 
be the polynomial ring over $k$ in $n+s$ variables and let $L_{n,s}:=\Frac R_{n,s}$ be its field of fractions.

\begin{definition}  \label{def:universal-relation}
A universal relation in Milnor K-theory modulo $m$ over the field $k$ is an identity
\begin{align} \label{eq:universal-rel}
(x_1,\dots ,x_n)=\lambda\cdot (a_1,\dots ,a_n) \in K^M_n(L_{n,s})/m ,
\end{align}
for some nonzero polynomials $a_1,\dots ,a_n\in R_{n,s}$ 
 and $\lambda\in( \Z/m)^\ast$.
\end{definition} 

\subsection{General properties}

\begin{lemma} \label{lem:universal-rel} 
Let 
(\ref{eq:universal-rel})
be a universal relation in Milnor K-theory modulo $m$ over the field $k$.
Let $L$ be a field extension of $k$ and let $\phi:R_{n,s}\to L$ be a morphism of $k$-algebras  
such that $\phi(x_i)$ and $\phi(a_i)$ are invertible in $L$ 
for all $i=1,\dots ,n$.
Then,
\begin{align*}
(\phi(x_1),\dots ,\phi(x_n))=\lambda\cdot (\phi(a_1),\dots ,\phi(a_n)) \in 
K^M_n(L)/m .
\end{align*}
\end{lemma}

\begin{proof} 
The morphism $\phi$ yields a morphism of schemes $\varphi:\Spec L\to \Spec R_{n,s}=\A_k^{n+s}$.
Let $x\in \A_k^{n+s}$ be the image of $\varphi$.
Then the field $L$ is an extension of the residue field $\kappa(x)$ of $x$ and so there is a natural homomorphism $K^M_n(\kappa(x))/m\to K^M_n(L)/m$.
In order to prove the lemma, we may thus without loss of generality assume $L=\kappa(x)$ and so $\varphi$ denotes the inclusion of the scheme-point $x\in \A_k^{n+s}$.  
 
Let $A$ be the local ring of $\A_k^{n+s}$ at $x$.
Since $\phi(x_i),\phi(a_i)\in L^\ast$, we get $x_i,a_i\in A^\ast$ and so
\begin{align}\label{eq:difference}
(x_1,\dots ,x_n)-\lambda (a_1,\dots ,a_n)\in K^M_n(A)/m .
\end{align}
This element lies in the kernel of $K^M_n(A)/m\to K^M_n(\Frac A)/m $, because $\Frac A=L_{n,s}$ and (\ref{eq:universal-rel}) is a universal relation. 
It thus follows from Lemma \ref{lem:spec-prop} that (\ref{eq:difference}) lies in the kernel of $K^M_n(A)/m\to K^M_n(L)/m $, because $L=\kappa(x)$ is the residue field of the local ring $A$ by assumption.
This concludes the lemma.
\end{proof}

\begin{remark}
The above lemma explains the terminology in Definition \ref{def:universal-relation}.
Indeed, let (\ref{eq:universal-rel}) be a universal relation, let $L$ be any field extension of $k$ and let $(\chi_1,\dots ,\chi_n)\in K^M_n(L)/m$ be any symbol.
By the universal property of polynomial rings, we may then define $\phi:R_{n,s}\to L$  by setting $\phi(x_i)=\chi_i$ for $i=1,\dots ,n$ and $\phi(y_j)\in L$ arbitrary such that $\phi(a_i)\neq 0$ for all $i=1,\dots ,n$.
By  Lemma \ref{lem:universal-rel}, we get a relation
\begin{align*}
(\chi_1,\dots ,\chi_n)=\lambda\cdot (\phi(a_1),\dots ,\phi(a_n)) \in 
K^M_n(L)/m 
\end{align*}
in the Milnor K-theory of $L$, which involves the symbol $(\chi_1,\dots ,\chi_n)$.  
\end{remark}

The following proposition shows that universal relations allow us  to construct varieties whose function fields kill a given symbol in Milnor K-theory modulo $m$.

\begin{proposition} \label{prop:univ-rel}
Let (\ref{eq:universal-rel}) be a universal relation in Milnor K-theory modulo $m$ over the field $k$ in degree $n\geq 1$.
Let $L$ be a field extension of $k$ and let $\chi_1,\dots ,\chi_n\in L^\ast$.
Let  $s'$ be a positive integer and let $\phi:R_{n,s}\to L[y_1,\dots ,y_{s'}]$ a homomorphism of $k$-algebras with $\phi(x_i)=\chi_i$ for all $i=1,\dots ,n$.
Let $c\in L[y_1,\dots ,y_{s'}]$ be such that
$$
F:=c^m-\sum_{i=1}^n \phi(a_i) \in L[y_1,\dots ,y_{s'}]
$$
is irreducible and let $W$ be a projective model of $\{F=0\}\subset \A_L^{s'}$.
Assume that for each $i$, $\phi(a_i)$ is not a multiple of $F$, i.e.\ the restriction of $\phi(a_i)$ to $W$ is nonzero. 
Then
\begin{enumerate}[(a)]
\item $
(\chi_1,\dots ,\chi_n)
 \in \ker \left(K^M_n(L)/m\longrightarrow K^M_n(L(W))/m \right) .
$ \label{item:univ-rel:W}
\item  \label{item:univ-rel:Y}
Let $Y$ be a variety over $L$ and let $\iota:Y\to W$ be a morphism of $L$-varieties such that the image $\iota(\eta_Y)$ of the generic point of $Y$ lies in the regular locus of $W$.
Then
$$
(\chi_1,\dots ,\chi_n) \in \ker \left(K^M_n(L)/m\longrightarrow K^M_n(L(Y))/m \right) .
$$
\end{enumerate}
\end{proposition}
\begin{proof}
Since $F$ is irreducible, $W$ is integral and so it is regular at the generic point.
In particular, item (\ref{item:univ-rel:W}) is a special case of (\ref{item:univ-rel:Y}).
Nonetheless, we will prove (\ref{item:univ-rel:W}) first.
For this, we denote by $\overline {\phi(a_i)}$ the image of $a_i\in R_{n,s}$ in $L(W)$.
By our assumptions, $\overline {\phi(a_i)}\neq 0$ for all $i$.
Hence, 
$$
(  {\phi(x_1)},\dots , {\phi(x_n)})=\lambda\cdot (\overline {\phi(a_1)},\dots ,\overline{\phi(a_n)})\in K^M_n(L(W))/m
$$
by Lemma \ref{lem:universal-rel}, and so this class vanishes by Lemma \ref{lem:(b1,...,bn)=0} because $\sum_i \overline {\phi(a_i)}$ is an $m$-th power in $L(W)$ by the definition of $F$.
This proves item (\ref{item:univ-rel:W}) because $\phi(x_i)=\chi_i$ for all $i$.

To prove item (\ref{item:univ-rel:Y}), let $w=\iota(\eta_Y)\in W$ be the image of the generic point of $Y$.
Let $A=\mathcal O_{W,w}$ be the local ring of $W$ at $w$.
By assumption, $A$ is a regular local ring.
Since $\chi_i\in L^\ast \subset A^\ast$ for all $i$, 
$$
(\chi_1,\dots ,\chi_n) \in \ker(K^M_n(A)/m\longrightarrow K^M_n(L(W))/m)
$$
by item (\ref{item:univ-rel:W}) proven above.
Applying Lemma \ref{lem:spec-prop}, we then find
$$
(\chi_1,\dots ,\chi_n) \in \ker(K^M_n(A)/m\longrightarrow K^M_n(\kappa(w))/m) .
$$
Item (\ref{item:univ-rel:Y}) stated in the proposition follows from this because $L(Y)$ is a field extension of $\kappa(w)$ and so the natural map $K^M_n(L)/m\longrightarrow K^M_n(L(Y))/m$ factors through $K^M_n(\kappa(w))/m$.
This concludes the proof of the proposition.
\end{proof}

\subsection{Examples}
The simplest example of a universal relation modulo $m$ is given by 
$$
(x_1)=(x_1y_1^m)\in K_1^M(L_{1,1})/m.
$$
The next lemma allows to produce universal relations in Milnor K-theory mod $m$ in arbitrary degree by starting with a single relation in low degree.  

\begin{lemma} \label{lem:propagate-univ-rel}
Let $(x_1,\dots ,x_n)=\lambda\cdot(a_1,\dots ,a_n) \in K^M_n(L_{n,s})/m$ be a universal relation in degree $n$.
Then
$$
\left( x_1,\dots ,x_{n},x_{n+1}\right) =\lambda\cdot\left( a_1,\dots ,a_n,x_{n+1}\left(y_{s+1}^m- \sum_{i=1}^na'_i\right) \right) \in K_{n+1}^M(L_{n+1,2s+1})/m,
$$
is a universal relation in degree $n+1$, where $a'_i:=a_i(x_1,\dots ,x_n,y_{s+2},\dots ,y_{2s+1})$.
\end{lemma}
\begin{proof}
Since $(x_1,\dots ,x_n)=\lambda\cdot (a_1,\dots ,a_n) \in K^M_n(L_{n,s})/m$, we have  
$$
\left( x_1,\dots ,x_n,x_{n+1}\left(y_{s+1}^m- \sum_{i=1}^na'_i \right) \right) =\lambda\cdot 
\left( a_1,\dots ,a_n,x_{n+1} \left(y_{s+1}^m- \sum_{i=1}^na'_i \right) \right)
$$
in $ K^M_n(L_{n+1,2s+1})/m$.
The claim in the lemma is thus equivalent to
\begin{align} \label{eq:lem:propagate}
 \left( x_1,\dots ,x_n,y_{s+1}^m-\sum_{i=1}^na'_i \right) =0\in  K_n^M(L_{n+1,2s+1})/m .
\end{align}
Relabelling the $y$-coordinates in the universal relation $(x_1,\dots ,x_n)=\lambda\cdot (a_1,\dots ,a_n)$ shows by Lemma \ref{lem:universal-rel} that $(x_1,\dots ,x_n)=\lambda\cdot (a'_1,\dots ,a'_n)\in  K_n^M(L_{n+1,2s+1})/m$ and so (\ref{eq:lem:propagate}) is equivalent to
$$
\lambda\cdot  \left( a'_1,\dots ,a'_n,y_{s+1}^m-\sum_{i=1}^na'_i \right) =0\in  K_n^M(L_{n+1,2s+1})/m ,
$$
which holds by Lemma \ref{lem:(b1,...,bn)=0}.
This concludes the proof of the lemma.
\end{proof}

To illustrate the above result, start with the trivial relation
$
(x_1)=(x_1y_1^m) 
$
in degree one.
Applying the lemma, we arrive at the relation
$$
(x_1,x_2)=(x_1y_1^m,x_2y_2^m-x_1x_2y_3^m) \in K_2^M(L_{2,3})/m 
$$
in degree two.
Applying the lemma once again, we get the universal relation
$$
(x_1,x_2,x_3)=(x_1y_1^m, x_2y_2^m-x_1x_2y_3^m, x_3y_4^m  - x_1x_3y_5^m  -x_2x_3y_6^m +x_1x_2x_3y_7^m ) 
$$
in $K_3^M(L_{3,7})/m$.
Repeating this process inductively, we are led to the universal relation in degree $n$ from Proposition \ref{prop:x1,...,xn} below.

\section{Fermat--Pfister forms}
\label{sec:fermat-pfister}
 
Let $k$ be a field and $m\geq 2$ an integer. 
For $n\geq 1$, we define the $n$-th Fermat--Pfister form of degree $m$ (and with coefficients in the polynomial ring $k[x_1,\dots ,x_n]$) as 
\begin{align} \label{def:Pfister}
\Pf_{m,n}(y_0,\dots ,y_{2^n-1}):= \sum_{\epsilon\in \{0,1\}^n}(-x_1)^{\epsilon_1}(-x_2)^{\epsilon_2}\dots (-x_n)^{\epsilon_n} \cdot y_{\phi(\epsilon)}^m,
\end{align}
where $\phi:\{0,1\}^n\to \{0,1,\dots ,2^n-1\}$ denotes the  bijection given by
$$
\phi(\epsilon)=  \sum_{i=1}^n \epsilon_i \cdot 2^{i-1} .
$$
This generalizes the famous quadratic forms of Pfister \cite{pfister} to higher degrees.\footnote{After this paper was submitted, it was brought to my attention that such generalizations of Pfister forms may already be found in a paper by Krashen and Matzri \cite[Proposition 1.9]{KM15}.}
We denote the coefficient in front of $y_j$ by $c_j$ 
and get
\begin{align} \label{def:ci}
\Pf_{m,n}(y_0,\dots ,y_{2^{n}-1})=\sum_{j=0}^{2^{n}-1} c_jy_j^m.
\end{align}
By definition, $c_0=1$, $c_1=-x_1$ and $c_{2^n-1}=(-1)^nx_1\cdots x_n$.

For $n\geq 1$, we have 
\begin{align*} 
\Pf_{m,n}(y_0,\dots ,y_{2^n-1})=\Pf_{m,n-1}(y_0,\dots ,y_{2^{n-1}-1})-
x_n\cdot \Pf_{m,n-1}(y_{2^{n-1}},\dots ,y_{2^{n}-1}) ,
\end{align*}
where we set $\Pf_{m,0}(y_0):=y_0^m$.
Inductively, this yields 
\begin{align}\label{eq:f_n:reciprocal}
\Pf_{m,n}(y_0,\dots ,y_{2^n-1})=y_0^m - \sum_{i=1}^n a_i , 
\end{align}
where 
\begin{align} \label{def:ai}
a_i:=x_i \cdot \Pf_{m,i-1}(y_{2^{i-1}},\dots ,y_{2^{i}-1}).
\end{align}

\begin{proposition} \label{prop:x1,...,xn}
Let $k$ be a field and let $a_i\in  k[x_1,\dots ,x_i,y_1,\dots ,y_{2^i-1}]$
 be as in (\ref{def:ai}).  
Then, 
\begin{align*} 
(x_1,\dots ,x_n)=(a_1,\dots ,a_n) \in K_n^M(L_{n,2^{n}-1})/m ,
\end{align*} 
is a universal relation in Milnor K-theory modulo $m$ over $k$.
\end{proposition}

\begin{proof}
We aim to prove the proposition by induction on $n$.
For $n=1$, the proposition is saying that $(x_1)=(x_1y_1^m)$, which is clear.
We now assume that the proposition is proven for some $n\geq 1$ and we aim to prove it for $n+1$.
Applying Lemma \ref{lem:propagate-univ-rel} to the given universal relation in degree $n$, we obtain 
\begin{align} \label{eq:lem:x1,...,xn}
(x_1,\dots ,x_n,x_{n+1})=\left( a_1,\dots ,a_n,x_{n+1}\left( y_{2^{n}}^m-\sum_{i=1}^n a_i'\right)  \right) ,
\end{align}
in $K_n^M(L_{n+1,2^{n+1}-1})/m$,
where 
$$
a'_i=a_i(x_1,\dots,x_n,y_{2^n+1},\dots ,y_{2^{n+1}-1})=x_i \cdot \Pf_{m,i-1}(y_{2^n+2^{i-1}},\dots ,y_{2^n+2^{i}-1}) .
$$ 
The recursive relation (\ref{eq:f_n:reciprocal}) implies
$$
 y_{2^{n}}^m-\sum_{i=1}^n a_i' = \Pf_{m,n}(y_{2^{n}},y_{2^n+1},\dots ,y_{2^{n+1}-1})
$$
and so 
$$
x_{n+1}\left( y_{2^{n}}^m-\sum_{i=1}^n a_i'\right)=x_{n+1}\Pf_{m,n}(y_{2^{n}},y_{2^n+1},\dots ,y_{2^{n+1}-1}) =a_{n+1}
$$
by (\ref{def:ai}). 
Hence, (\ref{eq:lem:x1,...,xn}) simplifies to 
$$
(x_1,\dots ,x_n,x_{n+1})=\left( a_1,\dots ,a_n,a_{n+1}   \right)\in K_n^M(L_{n+1,2^{n+1}-1})/m,
$$
as we want.
This concludes the proposition. 
\end{proof}


As an interesting example, the above discussion 
allows us to prove the following result, c.f.\ the aforementioned paper  \cite[Proposition 1.9]{KM15} by Krashen and Matzri for a similar result, proven with different methods. 

\begin{corollary}\label{cor:Pfister-forms-intro} 
Let $n\geq 2$ be an integer and let $\chi_1,\dots ,\chi_n\in L$ be nonzero elements of a field $L$.
Consider the hypersurface $X_{\chi_1,\dots ,\chi_n}\subset \CP^{2^{n}-1}_L$ of degree $m$, given by 
$$
 \sum_{\epsilon\in \{0,1\}^n}(-\chi_1)^{\epsilon_1}(-\chi_2)^{\epsilon_2}\dots (-\chi_n)^{\epsilon_n} \cdot y_{\phi(\epsilon)}^m=0 ,
$$
where  
$
\phi(\epsilon)=  \sum_{i=1}^n \epsilon_i \cdot 2^{i-1} 
$. 
If $X_{\chi_1,\dots ,\chi_n}$ is integral (e.g.\ if $\frac{1}{m}\in L$), then
$$
(\chi_1,\dots ,\chi_n)\in \ker(K_n^M(L)/m\longrightarrow K^M_n(L(X_{\chi_1,\dots ,\chi_n}))/m) .
$$
\end{corollary}

\begin{proof} 
Let $k$ be the prime field of $L$ and consider the polynomial ring 
$$
R_{n,2^{n}-1}=k[x_1,\dots ,x_n,y_1,\dots ,y_{2^n-1}]
$$ 
from Section \ref{sec:universal-rel}.
Let $\phi:R_{n,2^{n}-1}\to L[y_0,y_1,\dots ,y_{2^n-1}]$ be the morphism of $k$-algebras, given by $\phi(x_i)=\chi_i$ and $\phi(y_j)=y_j$ for all $i$ and $j$.
Let further $a_i\in R_{n,2^n-1}$ be as in (\ref{def:ai}), so that the universal relation $(x_1,\dots ,x_n)=(a_1,\dots ,a_n)\in K^M_n(L_{n,2^n-1})/m$ holds by Proposition \ref{prop:x1,...,xn}.
By (\ref{eq:f_n:reciprocal}), the hypersurface $X_{\chi_1,\dots ,\chi_n}$ from Corollary \ref{cor:Pfister-forms-intro} is given by $F=0$ where
$$
F:=y_0^m-\sum_{i=1}^n \phi(a_i).
$$
By assumption, $X_{\chi_1,\dots ,\chi_n}$ is integral and so $F$ is irreducible.
Since $\phi(x_i)=\chi_i\in L^\ast$ and $\phi(a_i)\not \equiv 0 \mod F$, it thus follows from item (\ref{item:univ-rel:W}) in Proposition \ref{prop:univ-rel} that
$$
(\chi_1,\dots ,\chi_n) \in \ker (K^M_n(L)/m\longrightarrow K^M_n(L(X_{\chi_1,\dots ,\chi_n}))/m).
$$ 
This proves Corollary \ref{cor:Pfister-forms-intro}.
\end{proof}

Note that in Corollary \ref{cor:Pfister-forms-intro}, the integer $m$ is not assumed to be invertible in $L$.
Adding this assumption, $X_{\chi_1,\dots ,\chi_n}$ is automatically integral and in fact smooth over $L$ and we obtain the following stronger statement.

\begin{corollary}\label{cor:Pfister-forms} 
Let $n\geq 2 $ be an integer and let $L$ be a field in which $m$ is invertible and let $\chi_1,\dots ,\chi_n\in L^\ast$. 
Consider the smooth hypersurface $X_{\chi_1,\dots ,\chi_n}\subset \CP^{2^{n}-1}_L$ of degree $m$, given by  
$$
 \sum_{\epsilon\in \{0,1\}^n}(-\chi_1)^{\epsilon_1}(-\chi_2)^{\epsilon_2}\dots (-\chi_n)^{\epsilon_n} \cdot y_{\phi(\epsilon)}^m=0 ,
$$
where  
$
\phi(\epsilon)=  \sum_{i=1}^n \epsilon_i \cdot 2^{i-1} 
$. 
Let $Y$ be a variety over $L$ which admits a morphism $\iota:Y\to X_{\chi_1,\dots ,\chi_n}$ of $L$-varieties.
Then
$$
(\chi_1,\dots ,\chi_n)\in \ker \left( K^M_n(L)/m\longrightarrow K^M_n(L(Y))/m \right) .
$$
\end{corollary}
\begin{proof}
Let $k$ be the prime field of $L$ and recall $R_{n,s}=k[x_1,\dots ,x_n,y_1,\dots ,y_s]$ from section \ref{sec:universal-rel}.
Let $\phi:R_{n,2^n-1}\to L[y_0,\dots ,y_{2^n-1}]$ be the morphism of $k$-algebras with $\phi(x_i)=\chi_i$ and $\phi(y_j)=y_j$ for all $i=1,\dots ,n$ and $j=1,\dots ,2^n-1$.

Note that $W:=X_{\chi_1,\dots ,\chi_n}$ is defined by the Fermat--Pfister form $\Pf_{m,n}(y_0,\dots ,y_{2^n-1})$ of degree $m$ from (\ref{def:Pfister}), where $x_i$ is replaced by $\chi_i$ for $i=1,\dots ,n$.
Hence, by (\ref{eq:f_n:reciprocal}), 
$$
W= \left\lbrace y_0^m-\sum_{i=1}^n \phi(a_i)=0 \right\rbrace  \subset \CP_L^{2^n-1},
$$
where $a_i=x_i\Pf_{m,n}(y_{2^{i-1}},\dots ,y_{2^i-1})$. 
Recall also that $(x_1,\dots ,x_n)=(a_1,\dots ,a_n)\in K^M_n(L_{n,2^n-1})/m$ is a universal relation in Milnor K-theory modulo $m$ by Proposition \ref{prop:x1,...,xn}.

Since $\chi_i\neq 0$ for all $i$ and $m$ is invertible in $L$, $W=X_{\chi_1,\dots ,\chi_n}$ is smooth over $L$ by the Jacobi criterion.
In particular, the image $\iota(\eta_Y)\in W$ of the generic point of $Y$ lies in the regular locus of $W$ and so Corollary \ref{cor:Pfister-forms} follows from item (\ref{item:univ-rel:Y}) in Proposition \ref{prop:univ-rel}. 
\end{proof}

\section{Unramified cohomology via universal relations}


\begin{definition} \label{def:twisting-type}
Let $k$ be a field.
A homogeneous polynomial $g\in k[x_0,x_1,\dots ,x_n]$ is of twisting type modulo $m$ if for all $ i=0,1,\dots ,n$:
\begin{itemize}
\item $g$ contains the monomials $x_i^{\deg g}$ nontrivially;
\item $g$ is an $m$-th power modulo $x_i$.
\end{itemize} 
An inhomogeneous polynomial $b\in k[x_1,\dots ,x_n]$ is \emph{of twisting type modulo $m$} if its homogenization in $ k[x_0,x_1,\dots ,x_n]$ has this property.
\end{definition}

Note that the degree of a polynomial which is of twisting type modulo $m$ must be a multiple of $m$.
The following slightly technical lemma will be crucial.

\begin{lemma} \label{lem:def:twisting-type}
Let $b\in k[x_1,\dots ,x_n]$ be an inhomogeneous polynomial of twisting type modulo $m$.
Let  $x\in S^{(1)}$ be a codimension one point of some normal birational model $S$ of $\CP^n_k$.
Let $z\in \CP^n_k$ be the image of $x$ under the birational map $S\dashrightarrow \CP^n_k$ and assume that $z$ is the generic point of the intersection of $c\geq 1$ coordinate hyperplanes $\{x_{i_1}=\dots =x_{i_c}=0\}$ with $0\leq i_1<\dots <i_c\leq n$.
Then $b$ becomes an $m$-th power in the fraction field of the completion $\widehat{\mathcal O_{S,x}}$ of the local ring of $S$ at $x$.
\end{lemma}
\begin{proof}
Let $g\in k[x_0,x_1,\dots ,x_n]$ be the homogeneous polynomial of twisting type given by homogenization of $b$.
The existence of $z$ implies $c\leq n$ and so there is some index $0\leq i_0\leq n$ with $x_{i_0}(z)\neq 0$.
Let $b'$ be the inhomogeneous polynomial, given by setting $x_{i_0}=1$ in $g$.
Then  
\begin{align*}
b\left( \frac{x_1}{x_0},\frac{x_2}{x_0},\dots ,\frac{x_n}{x_0} \right) 
&=g \left( \frac{x_0}{x_0}, \frac{x_1}{x_0},\frac{x_2}{x_0},\dots ,\frac{x_n}{x_0} \right)\\
&=\left( \frac{x_{i_0}}{x_0} \right)^{\deg g}\cdot  g \left(\frac{x_0}{x_{i_0}},\frac{x_1}{x_{i_0}},\dots ,\frac{x_n}{x_{i_0}} \right)\\
&= \left( \frac{x_{i_0}}{x_0} \right)^{\deg g}\cdot  b'\left( \frac{x_0}{x_{i_0}}  ,\frac{x_1}{x_{i_0}},   \dots ,\widehat {\frac{x_{i_0}}{x_{i_0}}},\dots ,\frac{x_n}{x_{i_0}} \right) .
\end{align*}
Since $g$ is of twisting type modulo $m$, $\deg(g)$ is divisible by $m$ and so $b$ becomes an $m$-th power in $\Frac \widehat{\mathcal O_{S,x}}$ if and only if this holds for $b'$.
For this reason we may without loss of generality assume that $i_0=0$ and so $x_0(z)\neq 0$. 
In particular, the inhomogenization $b$ given by setting $x_0=1$ in $g$ is a rational function on $\CP^n$ that is regular locally at $z$ (e.g.\ on the affine piece $\{x_0\neq 0\}$ which contains the point $z$).
That is, $b$ is contained in the local ring $\mathcal O_{\CP^n,z}\subset  k(\CP^n)=k(x_1,\dots, x_n)$.

Since $g$ contains the monomials $x_i^{\deg g}$ nontrivially for all $i=0,\dots ,n$, we have 
$$
b(0,0,\dots ,0)\neq 0
$$ 
and so $b$ does not vanish at $z$, because $z$ is the generic point of an intersection of $c\geq 1$ hyperplanes.
That is, the image $\overline b$ of $b$ in $\kappa(z)$ is nontrivial.
Moreover, $\overline b$ is an $m$-th power, as $g$ is an $m$-th power modulo $x_i$ for all $i$ and $c\geq 1$ by assumption.
The result thus follows from Hensel's lemma, applied to $\widehat{\mathcal O_{S,x}}$.
\end{proof}

For $n=2$ and $m=2$, the equation of a conic tangent to the three coordinate lines in $\CP^2$ is of twisting type, see \cite{HPT}.
An instructive example for arbitrary $m$ and $n$ is given by
\begin{align} \label{eq:g-twisting-type}
g=G^m+x_0^{em-n}x_1\cdots x_n ,
\end{align}
where $G$ is homogeneous of degree $e$ with $em>n$ and $G$ contains $x_i^e$ nontrivially for all $i=0,1,\dots ,n$.
For $m=2$, this simple but flexible example was used very successfully in \cite{Sch-JAMS}.
The general idea of tangentially meeting degeneracy loci goes back to Artin--Mumford \cite{artin-mumford} and has since then been used by many authors, see e.g.\ \cite{CTO,Pirutka,Sch-duke}.

\begin{theorem} \label{thm:unramified-coho}
Let $m\geq 2$, $n,s\geq 1$ be integers and let $k$ be an algebraically closed field in which $m$ is invertible.
Let $(x_1,\dots ,x_n)=\lambda \cdot (a_1,\dots ,a_n)\in K_n^M(L_{n,s})/m$ be a universal relation in Milnor K-theory modulo $m$ over $k$ and let $b\in k[x_1,\dots ,x_n]$ be an inhomogeneous polynomial of twisting type modulo $m$, see Definitions \ref{def:universal-relation} and \ref{def:twisting-type}.
Assume that the polynomial
\begin{align} \label{eq:F:thm:unramified}
F:=b-\sum_{i=1}^n a_i\in R_{n,s}=k[x_1,\dots ,x_n,y_1,\dots ,y_{s}] 
\end{align}
is irreducible and let $W$ be a projective model of $\{F=0\}\subset \A^{n+s}_k$ such that projection to the $x_i$-coordinates induces a morphism $h:W\to \CP^n_k$.
Let $Y$ be a projective variety over $k$ together with a morphism $\iota:Y\to W$, such that
\begin{itemize}
\item the image $\iota(\eta_Y)$ of the generic point of $Y$ lies in the smooth locus of $W$;
\item the composition $f:=h\circ \iota:Y\longrightarrow \CP^n_k$ is surjective.
\end{itemize} 
Then 
$$
\alpha:=(x_1,\dots ,x_n)\in H^n(k(\CP^n),\mu_m^{\otimes n})
$$
has the following properties.
\begin{enumerate}
\item The pullback $
f^\ast \alpha \in H_{nr}^n(k(Y)/k,\mu_m^{\otimes n})
$
is unramified over $k$. \label{item:unramified}
\item  \label{item:vanishing}
For any generically finite dominant morphism of $k$-varieties $\tau:Y'\to Y$ and any subvariety $E\subset Y'$ which meets the smooth locus of $Y'$ and which does not dominate $\CP^n$ via $f\circ \tau$, we have
$$
(\tau^\ast f^\ast \alpha )|_E=0\in H^n(k(E),\mu_m^{\otimes n}) .
$$
\item 
Assume that there is a discrete valuation ring $R\subset k$ with residue field $\kappa$ and a proper flat $R$-scheme $\mathcal Y\to \Spec R$ with $Y\cong \mathcal Y\times_R k$.
Assume further that $f:Y\to \CP^n_k$ extends to a morphism $f_R:\mathcal Y\to \CP^n_R$ whose base change $f_0:Y_0:=\mathcal Y\times_R\kappa \to \CP^n_\kappa$ to the special point of $\Spec R$ admits a rational section $\xi:\CP^n_\kappa \dashrightarrow Y_0$ whose image lies generically in the smooth locus of $Y_0$ over $\CP^n_\kappa$. 
Then $f^\ast \alpha\in H_{nr}^n(k(Y)/k,\mu_m^{\otimes n})$ has order $m$, i.e., $e\cdot f^\ast \alpha\neq 0$ for all $e=1,2,\dots ,m-1$.\label{item:nonzero}
\end{enumerate}
\end{theorem}

\begin{proof} 
Since $E\subset Y'$ in item (\ref{item:vanishing}) meets the smooth locus of $Y'$, we may without loss of generality assume that $Y'$ is normal.
Replacing $Y$ by its normalization, we may then assume that $Y$ is normal as well (because $\tau:Y'\to Y$ factors through the normalization of $Y$, once $Y'$ is normal).
By the same argument as at the beginning of the proof of \cite[Proposition 5.1]{Sch-JAMS}, item (\ref{item:unramified}) and (\ref{item:vanishing}) follow if we can show that for any codimension one point $y\in Y^{(1)}$, which does not map to the generic point of $\CP^n_k$, 
\begin{align} \label{eq:f*alpha=0}
\del_y(f^\ast \alpha)=0 \in H^{n-1}(\kappa(y),\mu_m^{\otimes (n-1)})\ \ \text{and}\ \ (f^\ast \alpha)|_y=0\in H^n(\kappa(y),\mu_m^{\otimes n}) .
\end{align} 

To prove (\ref{eq:f*alpha=0}), let us fix $y\in Y^{(1)}$ as above and let $c$ denote the number of coordinate hyperplanes $\{x_i=0\}\subset \CP^n_k$ which contain the point $f(y)$. 
By \cite[Proposition 1.6]{merkurjev}, we may also choose a normal birational model $S$ of $\CP^n_k$, such that $y$ maps via the induced rational map $Y\dashrightarrow S$ to a codimension one point $x\in S^{(1)}$ on $S$.

Let us first assume that $f(y)\in \CP^n_k$ has codimension $c$.
Then $f(y)$ must be the generic point of an intersection of $c$ coordinate hyperplanes.
(In particular, we have $c\geq 1$, because $f(y)$ is not the generic point of $\CP^n_k$.) 
Since $b$ is of twisting type modulo $m$, it follows from Lemma \ref{lem:def:twisting-type} that 
$b$ becomes an $m$-th power in the fraction field $L:=\Frac \widehat{\mathcal O_{S,x}}$ of the completion $\widehat {\mathcal O_{S,x}}$ of the local ring of $S$ at $x$.

Let $Y_\eta$ and $W_\eta$ be the generic fibres of $f:Y\to \CP^n_k$ and $h:W\to \CP^n_k$, respectively.
These are varieties over the field $k(\CP^n)$.
Since $L$ is a field extension of $k(\CP^n)$, we can consider the $L$-varieties
$$
(Y_\eta)_L:=Y_\eta\times_{k(\CP^n)}L\ \ \text{and}\ \ (W_\eta)_L:=W_\eta\times_{k(\CP^n)}L .
$$
Since $b=(b')^m$ for some $b'\in L$, we find that $(W_\eta)_L$ is birational to
$$
\left\lbrace (b')^m-\sum_{i=1}^na_i=0 \right\rbrace \subset \A^{s}_L .
$$

The morphism $\iota:Y\to W$ induces a morphism $(\iota_\eta)_L:(Y_\eta)_L\to (W_\eta)_L$ and the image of the generic point of $(Y_\eta)_L$ lies in the smooth locus of $(W_\eta)_L$, by assumption.
Since $(x_1,\dots ,x_n)=\lambda\cdot (a_1,\dots ,a_n)\in K^M_n(L_{n,s})/m$ is a universal relation modulo $m$ over $k$, we thus deduce from Proposition \ref{prop:univ-rel}, applied to the natural morphism $\phi:R_{n,s}\to L[Y_\eta]$, induced by $k[x_1,\dots ,x_n]\subset L$, that
\begin{align} \label{eq:LY_eta}
(x_1,\dots,x_n)\in \ker \left( K^M_n(L)/m\longrightarrow K^M_n(L(Y_\eta))/m \right). 
\end{align}

Let now $\widehat{\mathcal O_{Y,y}}$ be the completion of $Y$ at the codimension one point $y$.
Then the fraction field $\Frac  \widehat {\mathcal O_{Y,y}}  $ is a field extension of $L(Y_\eta)$ and so  (\ref{eq:LY_eta}) implies
$$
(x_1,\dots,x_n)\in \ker \left( K^M_n(L)/m\longrightarrow K^M_n\left( \Frac  \widehat {\mathcal O_{Y,y}} \right) /m \right) .
$$ 
Mapping this identity to cohomology via (\ref{eq:bass-tate}), we find that $f^\ast \alpha$ lies in the kernel of the natural map
$$
\varphi:H^n(k(Y),\mu_m^{\otimes n})\longrightarrow H^n(\Frac \widehat {\mathcal O_{Y,y}},\mu_m^{\otimes n}) .
$$
The residue of $f^\ast  \alpha$ at $y$ factors through $\varphi$, and so $\del_y f^\ast \alpha=0$.
This implies 
$$
\varphi(f^\ast  \alpha)=0\in H^n_{\text{\'et}}(\Spec \widehat {\mathcal O_{Y,y}},\mu_m^{\otimes n})\subset H^n(\Frac \widehat {\mathcal O_{Y,y}},\mu_m^{\otimes n}),
$$
where the latter inclusion follows from Theorem \ref{thm:ses}.
Hence, the restriction $(f^\ast \alpha)|_{y}$ factors through $\varphi$ as well and so $(f^\ast \alpha)|_{y}=0$, which concludes (\ref{eq:f*alpha=0}) in this case.

It remains to deal with the case where $f(y)\in \CP^n_k$ has codimension greater than $c$ (e.g.\ this happens if $c=0$). 
Using homogeneous coordinates, we have
$\alpha=\left( \frac{x_1}{x_0},\dots ,\frac{x_n}{x_0} \right)$.
Fix some $j\in\{ 1,\dots ,n\}$.
Multiplying each entry of $\alpha$ by $(x_0/x_j)^m$, we find
\begin{align*}
\alpha=\left( \frac{x_0^{m-1}x_1}{x_j^m},\dots,\frac{x_0^{m-1}}{x_j^{m-1}},\dots ,\frac{x_0^{m-1}x_n}{x_j^m} \right) .
\end{align*}
Since $k$ is algebraically closed, $(-1)\in (K^\ast)^m$ and so $(a,a)=0$ for any $a\in k(\CP^n)^\ast$, see Lemma \ref{lem:(b1,...,bn)=0}.
Applying this to $a=(x_0/x_j)^{m-1}$, the above identity simplifies to
\begin{align*}
\alpha=\left( \frac{x_1}{x_j},\dots,\frac{x_0^{m-1}}{x_j^{m-1}},\dots ,\frac{x_n}{x_j} \right)
=-\left( \frac{x_1}{x_j},\dots,\frac{x_0}{x_j},\dots ,\frac{x_n}{x_j} \right) .\end{align*}
Since it suffices to prove (\ref{eq:f*alpha=0}) after changing the sign of $\alpha$, we may thus, up to relabelling the coordinates, without loss of generality assume that $x_1,\dots ,x_c$ vanish at $f(y)$, while $x_0,x_{c+1},\dots ,x_n$ do not vanish at $f(y)$.

Now the same argument as in Case 2 of the proof of \cite[Proposition 5.1]{Sch-JAMS} applies; we repeat it for convenience of the reader.

First recall the normal birational model $S$ of $\CP^n_k$, such that $y$ maps to a codimension one point $x\in S^{(1)}$ on $S$.
Since $x_0,x_{c+1},\dots ,x_n$ do not vanish at $f(y)$, we get
$$
\del_x\alpha=\left( \del_x\left(x_1,\dots ,x_c \right)\right) \cup (x_{c+1},\dots ,x_n)\in H^{n-1}(\kappa(x),\mu_m^{\otimes (n-1)}),
$$
see e.g.\ \cite[Lemma 2.1]{Sch-JAMS}.
Since $f(y)$ has codimension greater than $c$ and $k$ is algebraically closed, $H^{n-c}(\kappa(f(y)),\mu_m^{\otimes j})=0$ for all $j$. 
Hence, $(x_{c+1},\dots ,x_n)=0\in H^{n-c}(\kappa(f(y)),\mu_m^{\otimes (n-c)})$ and so $\del_x\alpha=0$ by the above formula.
Since $\del_y\alpha$ is up to a multiple given by the pullback of $\del_x\alpha$ (see e.g.\ \cite[Proposition 3.3.1]{CT}), we find that $\del_yf^\ast \alpha=0$.
Moreover, the restriction $f^\ast\alpha|_{y}$ is given by pulling back the restriction $\alpha|_{x}\in H^n(\kappa(x),\mu_m^{\otimes n})$, which vanishes because $\kappa(x)$ has cohomological dimension less than $n$, since $k$ is algebraically closed.
This proves (\ref{eq:f*alpha=0}), which establishes items (\ref{item:unramified}) and (\ref{item:vanishing}) of Theorem \ref{thm:unramified-coho}. 

To prove (\ref{item:nonzero}), we define for any given field $K$, the class
$$
\alpha_K:=(x_1,\dots ,x_n)\in H^n(K(\CP^n),\mu_m^{\otimes n}).
$$
In particular, $\alpha=\alpha_k$ in the notation of Theorem \ref{thm:unramified-coho}.
We then assume for a contradiction that 
for some $e\in \{1,2,\dots ,m-1\}$,
\begin{align} \label{eq:jf*alpha=0}
e\cdot f^\ast \alpha_k=0 \in H^n(k(Y),\mu_m^{\otimes n}) .
\end{align}  

We denote the fraction field of $R$ by $L:=\Frac R$ and consider the morphism $f_R:\mathcal Y\to \CP_R^n$ that extends $f$ by assumption. 
In a first step, we aim to reduce to the case where
\begin{align} \label{eq:jf*alpha=0-L}
e\cdot f_R^\ast \alpha_L=0 \in H^n(L(Y),\mu_m^{\otimes n}) .
\end{align}  
(Here, by slight abuse of notation, we use that the $k$-variety $Y$ can be thought of as a variety over the smaller field $L\subset k$, because  $Y\cong \mathcal Y\times_Rk$ and $L=\Frac R$.) 

Since $k$ is algebraically closed, $H^n(k(Y),\mu_m^{\otimes n})\to H^n(K(Y),\mu_m^{\otimes n})$ is injective for any field extension $K$ of $k$.
(To see this, it suffices to treat the case where $K$ is a finitely generated field extension of $k$, hence the fraction field of an affine $k$-variety $B$ and so the claim follows after specialization to a $k$-point of $B$, which exists because $k$ is algebraically closed.)
We may thus without loss of generality assume that in (\ref{eq:jf*alpha=0}), $k$ is the algebraic closure of $\Frac R$.
Note also that the assumptions in item (\ref{item:nonzero}) of Theorem \ref{thm:unramified-coho} are stable under base change via an extension of discrete valuation rings $R\subset R'$. 
Replacing $R$ by its completion $\hat R$, $\mathcal Y\to \Spec R$ by the corresponding base change and $k$ by the algebraic closure of $\hat R$, we may thus assume that $R$ is complete.
Since $H^n(k(Y),\mu_m^{\otimes n})$ is the direct limit $\lim_{L'} H^n(L'(Y),\mu_m^{\otimes n})$, where $L'$ runs through all finitely generated extensions of $L=\Frac R$, (\ref{eq:jf*alpha=0-L}) shows that there is a finitely generated field extension $L'$ of $\Frac R$ such that $e\cdot f_R^\ast \alpha $ maps to zero on $ H^n(L'(Y),\mu_m^{\otimes n}) $. 
Since $k$ is the algebraic closure of $\Frac R$, $L'$ is in fact a finite field extension of $\Frac R$.
Replacing $R$ by its integral closure in $L'$ (which is again a discrete valuation ring because $R$ is complete, see \cite[Th\'eor\`eme 23.1.5 and Corollaire 23.1.6]{EGAIV}), $\mathcal Y\to \Spec R$ by the corresponding base change and $\kappa$ by the induced finite field extension, we may finally assume that (\ref{eq:jf*alpha=0-L}) holds, as we want.

By assumption, there is a rational section $\xi:\CP^n_\kappa \to Y_0$ such that the image $y_0=\xi(\eta_{\CP^n_\kappa})$ of the generic point of $\CP^n_\kappa$ is contained in the smooth locus of $Y_0$ over $\CP^n_\kappa$ (and hence in particular in the smooth locus of $Y_0$).
Since $R$ is a discrete valuation ring and $Y_0$ is the special fibre of the proper flat morphism $\mathcal Y\to \Spec R$, we find that $y_0$ is contained in a unique irreducible component $Y'_0$ of $Y_0$ and $Y_0$ must be generically reduced along $Y_0'$.
In particular, the local ring $A:=\mathcal O_{\mathcal Y,\eta_{Y_0'}}$ of $\mathcal Y$ at the generic point of $Y_0'$ is a discrete valuation ring with fraction field $L(Y)$, where we recall that $L=\Frac R$.
The morphism $f_R:\mathcal Y\to \CP^n_R$ induces a chain of inclusions $L(\CP^n)\subset A\subset L(Y)$.
Using this, the pullback map $f_R^\ast:H^n(L(\CP^n),\mu_m^{\otimes n})\to H^n(L(Y),\mu_m^{\otimes n})$ factors through $H^n_{\text{\'et}}(\Spec A,\mu_m^{\otimes n})$.
It thus follows from the vanishing in (\ref{eq:jf*alpha=0-L}) and the specialization property in Lemma \ref{lem:spec-prop-etale} that the image of $e\cdot \alpha_L$ in  $H^n_{\text{\'et}}(\Spec A,\mu_m^{\otimes n})$ restricts to zero on the special point of $\Spec A$. 
That is, 
\begin{align} \label{eq:e*f*_0alpha=0}
e\cdot f_0^\ast \alpha_\kappa = 0 \in H^n(\kappa(Y'_0),\mu_m^{\otimes n}) ,
\end{align}
where $f_0:Y_0\to \CP_\kappa^n$ denotes the base change of $f_R:\mathcal Y\to \CP^n_R$ over $R$ to $\kappa$.

Let $B$ be the local ring of $Y_0$ at the generic point $y_0=\xi(\eta_{\CP^n_\kappa})$  of the image of the section $\xi:\CP^n_\kappa \dashrightarrow Y_0$.
Since the image of $\xi$ is generically contained in the 
smooth locus of $Y_0$ over $\CP^n_\kappa$, $y_0$ is a smooth point of the generic fibre of $Y_0\to \CP^n_\kappa$.
This implies that
$B$ is a regular local ring with fraction field $\kappa(Y_0')$, where we recall that $Y'_0$ is the unique component of $Y_0$ that contains $y_0$.
The morphism $f_0:Y_0\to \CP^n_\kappa$  
thus induces a chain of inclusions $\kappa(\CP^n)\subset B\subset \kappa(Y'_0)$.
Using this, the pullback map $f_0^\ast:H^n(\kappa(\CP^n),\mu_m^{\otimes n})\to H^n(\kappa(Y'_0),\mu_m^{\otimes n})$ factors through $H^n_{\text{\'et}}(\Spec B,\mu_m^{\otimes n})$.
The vanishing in (\ref{eq:e*f*_0alpha=0}) and the specialization property in Lemma \ref{lem:spec-prop-etale} thus imply that the image of $e\cdot \alpha_\kappa$ in $H^n_{\text{\'et}}(\Spec B,\mu_m^{\otimes n})$ restricts to zero on the special point $\Spec \kappa(y_0)$ of $\Spec B$.
However, the composition
$$
H^n(\kappa(\CP^n),\mu_m^{\otimes n})\longrightarrow H^n_{\text{\'et}}(\Spec B,\mu_m^{\otimes n})\longrightarrow H^n(\kappa(y_0),\mu_m^{\otimes n}) ,
$$
is an isomorphism, because  $y_0=\xi(\eta_{\CP^n_\kappa})$ is the generic point of the image of the rational section $\xi:\CP^n_\kappa \dashrightarrow Y_0$ of $f_0:Y_0\to \CP^n_\kappa$.
Hence, the aforementioned vanishing in $H^n(\kappa(y_0),\mu_m^{\otimes n})$ implies that 
$$
e\cdot \alpha_\kappa=0\in H^n(\kappa( \CP^n),\mu_m^{\otimes n}).
$$ 
Since $e\in \{1,2,\dots ,m-1\}$ and $\alpha_\kappa=(x_1,\dots ,x_n)$, this statement is false, as one shows by induction on $n$ by taking the residue along $x_n=0$. 
This contradicts (\ref{eq:jf*alpha=0}), which completes the proof of the theorem.
\end{proof}

\begin{remark}
Starting with any universal relation $(x_1,\dots ,x_n)=\lambda\cdot (a_1,\dots ,a_n)\in K_n^M(L_{n,s})/m$,
Theorem \ref{thm:unramified-coho} produces hypersurfaces in $\A^{n+s}_k$ with nontrivial unramified $\Z/m$-cohomology whose degree is roughly the maximum of $m \lfloor \frac{n+1}{m}\rfloor$ (the degree of $g$ in (\ref{eq:g-twisting-type})) and the degrees of the $a_i$.
One source of examples for universal relations is given by Lemma \ref{lem:propagate-univ-rel}, but the notion is much more general than that.
For instance, if $a_i$ for $i=1,\dots ,n$ is as in (\ref{def:ai}), then a similar argument as in Lemma \ref{lem:propagate-univ-rel} 
shows that for any $1\leq i\leq n$:
$$
(x_1,\dots ,x_n)=\left( a_1,\dots ,a_{i-1},a_i\cdot \left( y_{2^n}^m-\sum_{j=1}^{i-1} a_j'\right) ,a_{i+1},\dots ,a_n\right) \in K_n^M(L_{n,2^n+2^i-1})/m
$$
where $a'_j=a_j(x_1,\dots ,x_n,y_{2^n+2^{j-1}},\dots ,y_{2^n+2^j-1})$.
Since $\deg a_i=i+m$, the maximum of the degrees of the entries in the above relation coincides with those of $(x_1,\dots ,x_n)=(a_1,\dots ,a_n)$ from Proposition \ref{prop:x1,...,xn} as long as $i\leq (n-m+1)/2$, but the number of $y$-variables involved in the above relation is larger.
We will however not be able to use such relations in the proof of Theorem \ref{thm:torsion-order}, because the hypersurface in $\CP^{2^n+2^i-1}_K$ over $K=k(x_1,\dots ,x_n)$ given by the projective closure of $F=0$ with $F$ as in (\ref{eq:F:thm:unramified}) is not smooth over $K$.
\end{remark}

\section{Degeneration}

The following proposition generalizes \cite[Proposition 3.1]{Sch-JAMS} to degenerations with reducible special fibres.
The result is a variant of the author's improvement \cite{Sch-duke} of the method of Voisin \cite{voisin} and Colliot-Th\'el\`ene--Pirutka \cite{CTP}.
The original method of Voisin and Colliot-Th\'el\`ene--Pirutka had been generalized to degenerations with reducible special fibres by Totaro \cite{totaro}.

\begin{proposition}\label{prop:degeneration}
Let $R$ be a discrete valuation ring with fraction field $K$ and algebraically closed residue field $k$.
Let $\mathcal X\to \Spec R$ be a proper flat $R$-scheme with geometric generic fibre $X_{\overline \eta}=\mathcal X\times_R \overline K$ and special fibre $X_0=\mathcal X\times_R k$. 
Assume that $X_{\overline \eta}$ is integral.
Let $Y\subset X_0^{\red}$ be an irreducible component of the reduction of $X_0$ and assume that $X_0$ is reduced at the generic point of $Y$.
Let $m\geq 2$ be an integer that is invertible in $k$ and let $\tau:Y'\to Y$ be an alteration whose degree is coprime to $m$. 
Suppose that for some integers $i\geq 1$ and $j$ there is a class $\gamma \in H^i_{nr}(k(Y)/k,\mu_m^{\otimes j})$ of order $m$  
 such that 
$$
(\tau^\ast\gamma)|_{E}=0\in H^i(k(E),\mu_m^{\otimes j})\ \ \text{for any subvariety $E\subset \tau^{-1}(Y\cap X_0^{\sing})$.}
$$ 
Then the torsion order of $X_{\overline \eta}$ is divisible by $m$. 
\end{proposition}

\begin{proof}
We may assume that $e:=\Tor(X_{\overline \eta})$ is finite.
Since torsion orders remain unchanged under passage from an algebraically closed field to a bigger field (see \cite[Lemma 1.11]{CL}), we may after replacing $R$ by its completion assume that $R$ is complete.
The decomposition of $e\cdot \Delta_{X_{\overline \eta}}$ in the Chow group of $X_{\overline \eta}\times X_{\overline \eta}$ holds already over a finite
field extension $L$ of $\Frac(R)$, and so $ X_{\eta}\times L$ has torsion order $e$ for some finite extension $L$ of $\Frac(R)$, where $X_\eta=\mathcal X\times_R K$ denotes the generic fibre.
Since $R$ is a complete discrete valuation ring, the integral closure $R'$ of $R$ in $L$ is again a complete discrete valuation ring, see \cite[Th\'eor\`eme 23.1.5 and Corollaire 23.1.6]{EGAIV}.
Replacing $R$ by the base change to $R'$, we may thus assume that the generic fibre $X_{\eta}$ has torsion order $e$ (note that this does not change the special fibre).

Let $A:=\mathcal O_{\mathcal X,y}$ be the local ring of $\mathcal X$ at the generic point $y\in \mathcal X$ of $Y$. 
Since $X_0$ is a Cartier divisor on $\mathcal X$ which is reduced at $y$, it follows that $\mathcal X$ is regular at $y$.
Hence, $A$ is a discrete valuation ring with fraction field $K(X_{\eta})$.
Let $\delta_{X_{\eta}}\in \CH_0(X_\eta \times K(X_\eta))$ be the class induced by the diagonal.
By assumption,
$$
e\cdot \delta_{X_{\eta}}=z\times K(X_{\eta}) \in \CH_0(X_\eta \times K(X_\eta))
$$
for a zero-cycle $z\in \CH_0(X_\eta)$ of degree $e$.
Applying Fulton's specialization map on Chow groups \cite[\S 20.3]{fulton} to the proper flat family $\mathcal X_A\to \Spec A$, given by base change of $\mathcal X\to \Spec R$, we find that
\begin{align}\label{eq:e*delta_Y}
e\cdot \delta_{Y}=z_0\times k(Y) \in \CH_0(X_0\times k(Y))
\end{align}
for some zero-cycle $z_0\in \CH_0(X_0)$ of degree $e$, where $\delta_Y$ denotes the class of the diagonal of $Y$.
Let $U\subset Y$ be the complement of $Y\cap X_0^{\sing}$.
Since $X_0$ is reduced at the generic point of $Y$, $U$ is a non-empty open subset of $Y$.
Let $U':=\tau^{-1}(U)\subset Y'$.
Note that $U$ is smooth by construction, and so we can pullback cycle classes (modulo rational equivalence) via $U'\to U$ (see \cite[\S 8]{fulton}).
Since $U\to X_0$ is an open embedding, it is flat and so we can pullback cycles (resp.\ cycle classes) via this map as well.
Altogether, we can pullback (\ref{eq:e*delta_Y}) to $U'\times k(Y)$ via the natural map $U'\times k(Y)\to X_0\times k(Y)$.  
Applying the localization exact sequence associated to $U'\subset Y'$, we get
\begin{align}\label{eq:e*tau*delta_Y}
e\cdot \tau^\ast \delta_{Y}=z_{Y'}\times k(Y)+z' \in \CH_0(Y'\times k(Y)),
\end{align}
for some zero-cycle $z_{Y'}\in \CH_0(Y')$ (not necessarily of degree $e$ anymore) and a zero-cycle $z'\in \CH_0(Y'\times k(Y))$ which is supported on
$$
(Y'\setminus U')\times k(Y)=\tau^{-1}(Y\cap X_0^{\sing})\times k(Y).
$$

The end of the proof is now as in \cite[Proposition 3.1]{Sch-JAMS}:
the pairing (see \cite[\S 2.4]{merkurjev}) of the unramified cohomology class $\tau^\ast \gamma\in H^i_{nr}(k(Y')/k,\mu_m^{\otimes j})$ with the right hand side of (\ref{eq:e*tau*delta_Y}) vanishes, because $\tau^\ast \gamma$ vanishes when restricted to closed points of $Y'$ (because $k=\overline k$) or to subvarieties of $\tau^{-1}(Y\cap X_0^{\sing})$ (by assumption), while the left hand side evaluates to 
$$
e\cdot  \langle\tau^\ast \delta_{Y}, \tau^\ast \gamma \rangle=e\cdot  \langle \tau_\ast \tau^\ast \delta_{Y},  \gamma \rangle=e\cdot \deg(\tau)\cdot \gamma\in H^i(k(Y),\mu_m^{\otimes j}).
$$
Hence, $e\cdot \deg(\tau)\cdot \gamma=0$.
Since $\gamma$ has order $m$ and $\deg (\tau)$ is coprime to $m$, this is only possible if $e$ is divisible by $m$, as we want.
This completes the proof.
\end{proof}

\section{Proof of main results} \label{sec:main-proofs}
 
 Theorem \ref{thm:torsion-order} stated in the introduction follows from the following slightly stronger statement.
 
 \begin{theorem}\label{thm:torsion-order-2}
Let $k$ be an uncountable field and let $m\geq 2$ be an integer that is invertible in $k$.
Let $N\geq 3$ be an integer and write $N=n+r$ with $2^{n-1}-2\leq r\leq 2^n-2$. 
Then the torsion order of a very general Fano hypersurface $X_d\subset \CP^{N+1}_{k}$ of degree $d\geq m+n$ is divisible by $m$. 
\end{theorem}

\begin{remark} \label{rem:thm:torsion-order-2}
The bounds on $r$ in Theorem \ref{thm:torsion-order-2} ensure that any integer $N\geq 3$ can be written uniquely as a sum $N=n+r$ as in the theorem.
In the proof of Theorem \ref{thm:torsion-order-2} below, only the upper bound on $r$ will be used, while the lower bound only appears for convenience as it yields the strongest results on the divisibility of the respective torsion orders in fixed dimension $N$.
\end{remark}

\begin{proof}
Replacing $k$ by its algebraic closure, we may assume that $k$ is algebraically closed.
Denote by
$x_0,\dots ,x_n,y_1,\dots ,y_{r+1}$ homogeneous coordinates of $\CP^{N+1}_k$.
Let $t\in k$ be transcendental over the prime field $k'$ of $k$ (if $\operatorname{char}(k)=0$, we may also take $t$ to be a prime number coprime to $m$),  
and consider 
\begin{align} \label{def:g:proof-of-thm}
g(x_0,\dots ,x_n):=t\cdot \left( \sum_{i=0}^n x_i^{\lceil \frac{n+1}{m} \rceil} \right) ^m-(-1)^{n}x_0^{m \lceil\frac{n+1}{m}\rceil -n}x_1x_2\dots x_n ,
\end{align}
which is a homogeneous polynomial  of degree 
$
\deg(g)=m\lceil \frac{n+1}{m}\rceil\leq m+n 
$
in $k[x_0,\dots ,x_n]$.
It follows directly from Definition \ref{def:twisting-type} that $g$ is of twisting type.


We first deal with the case $d=m+n$ in the statement of Theorem \ref{thm:torsion-order-2}.
Consider the hypersurface $Z:=\{F=0\}\subset \CP^{N+1}_k$ of degree $m+n$, given by 
$$
F:=g(x_0,\dots ,x_n)\cdot x_0^{m+n-\deg (g)}+\sum_{j=1}^r x_0^{n-\deg c_j}c_j(x_1,\dots ,x_n)y_j^m+(-1)^n x_1x_2\cdots x_ny_{r+1}^m ,
$$
where $c_i(x_1,\dots ,x_n)\in k[x_1,\dots ,x_n]$ for $1\leq i\leq 2^n-1$ denote the coefficients of the Fermat--Pfister form (\ref{def:ci}).
Since $g$ in (\ref{def:g:proof-of-thm}) is not divisible by $x_i$ for any $i$, one easily checks that the hypersurface $Z$ is integral.
Consider the $r$-plane $P:=\{x_0=x_1=\dots =x_n=0\}\subset \CP^{N+1}$ and let $Y:=Bl_PZ$.
This blow-up is a hypersurface in $Bl_P(\CP^{N+1})\cong \CP(\mathcal O_{\CP^n}(-1)\oplus \mathcal O_{\CP^{n}}^{\oplus (r+1)})$, given by the equation
\begin{align}\label{eq:Y}
g(x_0,\dots ,x_n)\cdot x_0^{m+n-\deg (g)}z_0^m+\sum_{j=1}^r x_0^{n-\deg c_j}c_j(x_1,\dots ,x_n)z_j^m+(-1)^n x_1x_2\cdots x_nz_{r+1}^m=0 ,
\end{align}
where $z_0$ is a local coordinate that trivializes $\mathcal O_{\CP^n}(-1)$ and $z_1,\dots ,z_{r+1}$ trivialize $ \mathcal O_{\CP^{n}}^{\oplus(r+1)}$.
In the above coordinates, the exceptional divisor $D\subset Bl_PZ$ is given by $z_0=0$.
Projection to the $x_i$-coordinates yields a morphism $f:Y\to \CP^n_k$ whose generic fibre $Y_\eta$ is the smooth hypersurface of degree $m$ and dimension $r+1$ over $K=k(x_1,\dots ,x_n)$, given by setting $x_0=1$ in (\ref{eq:Y}).

To emphasize the dependence on the integers $n$ and $r$, we write $Y=Y_{n,r}$ for the projective variety given by (\ref{eq:Y}).
Then  $Y_{n,r}\subseteq Y_{n,2^n-2}$ because $r\leq 2^n-2$.
We claim that Theorem \ref{thm:unramified-coho} applies to $Y=Y_{n,r}$ and $W=Y_{n,2^n-2}$. 

Recall that $\Pf_{m,n}(y_0,\dots ,y_{2^n-1})=\sum_{j=0}^{2^n-1} c_jy_j^m$ by (\ref{def:ci}) with $c_{2^n-1}=(-1)^nx_1\dots x_n$.
Setting $x_0=z_0=1$ in (\ref{eq:Y}), we thus see that  $W=Y_{n,2^n-2}$ is birational to the affine hypersurface, given by
$$
g(1,x_1,\dots ,x_n)+\Pf_{m,n}(0,y_1,y_2,\dots ,y_{2^n-1})=0 .
$$
By (\ref{eq:f_n:reciprocal}), the above equation can be rewritten as
$$
g(1,x_1,\dots ,x_n)-\sum_{i=1}^na_i=0 
$$
where $a_i=x_i\Pf_{m,i-1}(y_{2^{i-1}},\dots ,y_{2^i -1})$ is as in (\ref{def:ai}).
By Proposition \ref{prop:x1,...,xn}, we have the universal relation
$$
(x_1,\dots ,x_n)=(a_1,\dots ,a_n) \in K_n^M(L_{n,2^n-1})/m .
$$
Since the generic fibre of $W\to \CP^n_k$ is smooth and contains the image of the generic point of $Y\hookrightarrow W$, we conclude that item (\ref{item:unramified}) and (\ref{item:vanishing}) of Theorem \ref{thm:unramified-coho} apply to $Y=Y_{n,r}$ and $W=Y_{n,2^n-2}$.
To see that the assumptions of item  (\ref{item:nonzero})  in Theorem \ref{thm:unramified-coho} are satisfied as well, consider the discrete valuation ring $R=k'[t]_{(t)}\subset k$ with parameter $t$ and residue field $\kappa=k'$, where we recall that $k'$ is the prime field of $k$.
(If $\operatorname{char}(k)=0$ and $t=p$ is a prime number coprime to $m$, then we consider $R=\Z_p$ with residue field $\kappa=\Z/p$.)
Since $Y$ is defined by the equation (\ref{eq:Y}) whose coefficients are all contained in $R$ and whose reduction modulo $t$ is nonzero, it is immediate that $Y$ extends to a proper flat $R$-scheme $\mathcal Y\to \Spec R$, where $\mathcal Y$ is the hypersurface defined by (\ref{eq:Y}) inside the projective bundle $\CP \left( \mathcal O_{\CP_R^n}(-1)\oplus \mathcal O_{\CP_R^{n}}^{\oplus (r+1)} \right) $ over $\CP^n_R$.
Since the morphism $f:Y\to \CP^n_k$ is induced by projection to the $x_i$-coordinates,  $f$ extends to a morphism of $R$-schemes $f_R:\mathcal Y\to \CP^n_R$.
The reduction $Y_0:=\mathcal Y\times_R \kappa$ is given by the equation (\ref{eq:Y}) where $g(x_0,\dots ,x_n)$ from (\ref{def:g:proof-of-thm}) must be replaced by its reduction modulo $t$ (which is
$
-(-1)^{n}x_0^{m \lceil\frac{n+1}{m}\rceil -n}x_1x_2\dots x_n 
$).
Setting $x_0=1$ in this equation, 
 we obtain 
\begin{align} \label{eq:Y_0-x0=1}
 -(-1)^{n} x_1x_2\dots x_nz_0^m+\sum_{j=1}^r c_j(x_1,\dots ,x_n)z_j^m+(-1)^n x_1x_2\cdots x_nz_{r+1}^m=0 .
\end{align}
The morphism $f_0:Y_0\to \CP^n_\kappa$ is given by projection to the $x_i$-coordinates and (\ref{eq:Y_0-x0=1}) describes the preimage $f^{-1}_0(U)$ of
 the affine piece $U:= \CP^n_\kappa\setminus \{x_0=0\}$.
 From this description it is clear that $f_0$
 admits a rational section $\xi:\CP^n_\kappa\dashrightarrow Y_0$, defined by setting $z_0=z_{r+1}=1$ and $z_j=0$ for $1\leq j\leq r$ in (\ref{eq:Y_0-x0=1}).
 The generic fibre of $f_0$ is the hypersurface in $\CP^{r+1}_{k(x_1,\dots ,x_n)}$ given by (\ref{eq:Y_0-x0=1}).
Since $m$ is invertible in $k$, the generic fibre of $f_0$ is smooth and so $\xi(\eta_{\CP^n_\kappa})$ is contained in the smooth locus of $Y_0$ over $\CP^n_\kappa$.
Hence, the assumptions of item  (\ref{item:nonzero})  in Theorem \ref{thm:unramified-coho} are satisfied and we conclude that 
$$
f^\ast (x_1,\dots ,x_n)\in H^n_{nr}(k(Y)/k,\mu_m^{\otimes n})
$$
has order $m$. 

Recall that $Y=Bl_PZ$ is birational to the hypersurface $Z:=\{F=0\}\subset \CP^{N+1}_k$ from above and so the above unramified class yields a class
$$
\gamma\in H^n_{nr}(k(Z)/k,\mu_m^{\otimes n})
$$
of order $m$.
Let $\tau':Y'\to Y$ be an alteration of degree coprime to $m$ (which exists by \cite[Theorem 1.2.5]{temkin} because $m$ is invertible in $k$) and let $\tau:Y'\to Z$ be the induced alteration of $Z$ (which has the same degree as $\tau'$).
Let $E\subset Y'$ be a closed subvariety with $\tau(E)\subset Z^{\sing}$.
If $\tau'(E)\subset Y$ does not dominate $\CP^n_k$ via $f:Y\to \CP^n_k$, then 
$$
\tau^\ast \gamma|_E=0\in H^n(k(E),\mu_m^{\otimes n})
$$
by item (\ref{item:vanishing}) in Theorem \ref{thm:unramified-coho}.
Otherwise, the natural map $E\to \CP^n_k$ induced by $f\circ \tau'$ is surjective and we denote its generic fibre by $E_\eta$.
The alteration $\tau'$ induces a morphism 
$$
\tau'_\eta:E_\eta\longrightarrow Y_\eta,
$$
 where $Y_\eta$ denotes the generic fibre of $f:Y\to \CP_k^n$.
 Since $m$ is invertible in $k$, $Y_\eta$ is smooth.
 This implies that the generic point of $\tau'(E)$ is a smooth point of $Y$. 
 Since $\tau(E)$ lies in the singular locus of $Z$, this implies that the generic point of $\tau'(E)\subset Y$ must lie on the exceptional divisor $D$ of the blow-up $Y=Bl_PZ$.
Hence, 
$$
\tau'_\eta (E_\eta)\subset D_\eta ,
$$ 
where $D_\eta$ denotes the generic fibre of 
 $f|_D:D\to \CP^n_k$.  
As explained above, $D$ is given by setting $z_0=0$ in (\ref{eq:Y}) and so $D_\eta\subset \CP^{r}_{k(\CP^n)}$ is the hypersurface over $k(\CP^n)$, given by
$$
\sum_{j=1}^r  c_j(x_1,\dots ,x_n)z_j^m+(-1)^n x_1x_2\cdots x_nz_{r+1}^m=0 .
$$
Hence, $D_\eta$ is a subvariety of the hypersurface $X_{x_1,\dots ,x_n}\subset \CP^{2^n-1}_L$ over $L=k(\CP^n)$ from Corollary \ref{cor:Pfister-forms} and so the natural map $E_\eta\to D_\eta$ induced by $\tau'_\eta$ induces a morphism $\iota:E_\eta\to X_{x_1,\dots ,x_n}$ of $L$-varieties.
It thus follows from Corollary \ref{cor:Pfister-forms} that
$$
(x_1,\dots ,x_n)\in \ker(K^M_n(k(\CP^n))/m\longrightarrow K^M_n(k(\CP^n)(E_\eta))/m).
$$
Since $k(\CP^n)(E_\eta)=k(E)$, we conclude by mapping this via (\ref{eq:bass-tate}) to cohomology that
$$
\tau^\ast \gamma|_E=0\in  H^n(k(E),\mu_m^{\otimes n}) .
$$

Altogether, this shows that the hypersurface $Z\subset \CP^{N+1}_k$ of degree $d$ satisfies the assumption on the special fibre in the degeneration technique of Proposition \ref{prop:degeneration} and so any integral hypersurface which degenerates to $Z$ (in the sense of Section \ref{subsec:degeneration}) has torsion order divisible by $m$.
This applies in particular to very general hypersurfaces of degree $d=m+n$ in $\CP^{N+1}_k$ (see Section \ref{subsec:degeneration}), which concludes the proof in the case where $d=m+n$.

If $d>n+m$, then a very general hypersurface of degree $d$ in $\CP^{N+1}_k$ degenerates to the union of $Z$ from above with $\{x_0^{d-m-n}=0\}$.
Recall that $Y=Bl_PZ$ (explicitly given by (\ref{eq:Y})) admits a morphism $f:Y\to \CP^n_k$ given by projection to the $x_i$-coordinates.
In particular, the preimage $f^{-1}\{x_0^{d-m-n}=0\}\subset Y$ does not dominate $\CP^n_k$.
Using item (\ref{item:vanishing}) in Theorem \ref{thm:unramified-coho}, we thus conclude as before that for any subvariety 
$$
E\subset \tau^{-1}(Z^{\sing}\cup \{x_0^{d-m-n}=0\}),
$$ 
$\tau^\ast \gamma|_E=0$.
Hence, Proposition \ref{prop:degeneration} applies and we get $m\mid e$ as before.
This concludes the proof of the theorem.
\end{proof}

\begin{proof}[Proof of Theorem \ref{thm:asoks-question}]
Let us now assume that $k=\C$ and let $m,n\geq 2$ and $N \geq 3$ be integers with $\log_2(m+1)\leq n \leq N+1-m$.
We aim to construct a rationally connected smooth complex projective variety $X$ of dimension $N$ such that $H^n_{nr}(\C(X)/\C,\Z/m)$ contains an element of order $m$.
Since $\C$ contains a primitive $m$-th root of unity, we have $H^n_{nr}(\C(X)/\C,\Z/m)\cong H^n_{nr}(\C(X)/\C,\mu_m^{\otimes n})$ and so it suffices to construct an element of order $m$ in the latter group.
Replacing $X$ by a product with projective space, we see that it suffices to deal with the case where 
$$
N=n-1+m\ \ \text{and}\ \ 2^n\geq m+1.
$$

Let $r:=m-1$.
Then $r\leq 2^n-2$ and so, in view of  Remark \ref{rem:thm:torsion-order-2}, we may consider  the projective variety $Y=Y_{n,r}$ such that $H^n_{nr}(\C(Y)/\C,\mu_m^{\otimes n})$ contains an element of order $m$  from the proof of Theorem \ref{thm:torsion-order}.
There is a morphism $f:Y\to \CP^n$ whose generic fibre is a smooth hypersurface of degree $m$ in $\CP_{\C(\CP^n)}^{r+1}$, given by the equation (\ref{eq:Y}). 
Since $m= r+1$, a general fibre of $f$ is Fano and so it is rationally chain connected, see \cite{campana,KMM} or \cite[Theorem V.2.1]{kollar2}, and hence rationally connected because $k=\C$, see \cite{KMM} or \cite[Theorem IV.3.10]{kollar2}. 
It thus follows from the Graber--Harris--Starr theorem \cite{GHS} that any resolution $X$ of $Y$ is a rationally connected variety of dimension $N=n+r$.
Since $X$ is birational to $Y$, $H^n_{nr}(\C(X)/\C,\mu_m^{\otimes n})=H^n_{nr}(\C(Y)/\C,\mu_m^{\otimes n})$ contains an element of order $m$, as we want. 
This concludes the proof of Theorem \ref{thm:asoks-question}.
\end{proof}

\begin{remark} \label{rem:thm:asok:Q}
As indicated in the proof of Theorem \ref{thm:torsion-order-2}, if $\operatorname{char}(k)=0$, then $t$ can be chosen to be a prime number coprime to $m$.
With this choice, the example constructed in the proof of Theorem \ref{thm:asoks-question} will be defined over $\Q$.
\end{remark}

\section{Cyclic covers}  
Theorem \ref{thm:cyclic-cover:intro} follows from the following more precise result.

\begin{theorem} \label{thm:cyclic-cover}
Let $k$ be an uncountable field and let $m\geq 2$ be an integer that is invertible in $k$. 
Let $N\geq 3$ be an integer and write $N=n+r$ with $2^{n-1}-2\leq r\leq 2^n-2$. 
Consider a cyclic $m:1$ cover $X\to \CP^{N}_{k}$ branched along a very general hypersurface  of degree $d$ 
 with $m\mid d$. 
Assume $d\geq m (\lceil \frac{n+1}{m}\rceil +\epsilon)$, where $\epsilon\in \{1,2\}$ is given by
$$
\epsilon :=
\begin{cases}
1\ \ \text{if $m$ divides   $n$, $n-1$ or $n-2$;}\\
2\ \ \text{otherwise.}
\end{cases}
$$  
Then the torsion order of $X$ is divisible by $m$.  
\end{theorem}

\begin{proof}
Replacing $k$ by its algebraic closure, we may assume that $k$ is algebraically closed.
Let $x_0,\dots ,x_n,y_2,\dots ,y_{r+1}$ be homogeneous coordinates on $\CP^{N}_k$ (note that we left out $y_1$).  
Let $d\geq m (\lceil \frac{n+1}{m}\rceil +\epsilon)$ be an integer that is divisible by $m$.
Let $g\in k[x_0,\dots ,x_n]$ be the polynomial from (\ref{def:g:proof-of-thm}) and consider 
the hypersurface 
$$
B:=\{F=0\}\subset \CP^N_k
$$ 
of degree $d$, given by
$$
F:=x_1^{m-1}g\cdot x_0^{d-\deg (g)-m+1}+x_1^{m-1}\sum_{j=2}^r x_0^{d-2m+1-\deg c_j}c_jy_j^m+(-1)^n x_0^{d-m-n+1}x_2\dots x_ny_{r+1}^m ,
$$
where $c_j\in k[x_1,\dots ,x_n]$ is as in (\ref{def:ci}).
Here the sum in the middle is empty if $r=1$.

We claim that our lower bound on $d$ ensures that $F$ is indeed a polynomial, i.e.\ the exponents of $x_0$ are non-negative.
To see this, write $n=a+bm$ for integers $b\geq 0$ and $m>a\geq 0$.
Then $\deg(g)=m\lceil\frac{n+1}{m} \rceil =m(b+1)$ and the lower bound on $d$ reads as $d\geq m(b+1+\epsilon)$, where $\epsilon =1$ if $a\leq 2$ and $\epsilon =2$ otherwise. 
This implies
$$
d-\deg (g)-m+1=d-m(b+1)-m+1\geq 1+m(\epsilon -1)>0.
$$
Moreover, for $j\leq r$, $\deg(c_j)\leq n-1$ and so
\begin{align*}
d-2m+1-\deg c_j\geq d-2m+2-n &\geq m(b+1+\epsilon)-2m+2-a-bm\\
&=2-a+m(\epsilon -1)\geq 0,
\end{align*}
because $0\leq a<m$ and $\epsilon =1$ if $a\leq 2$ and $\epsilon =2$ otherwise.
Also, since $m\geq 2$, 
$$
d-m-n+1> d-2m+2-n\geq 0
$$
by the above estimate.
Altogether, this shows that $F$ is a polynomial, as we want.

Consider the linear subspace  $P:=\{x_0=x_1=\dots =x_n=0\}\subset \CP^N_k$ of codimension $n+1$.
The blow-up $Bl_P\CP^N_k$ admits a natural morphism $\pi:Bl_P\CP^N_k\to \CP^n_k$ that is induced by projection to the $x_i$-coordinates.
This morphism  identifies $Bl_P\CP^N_k$ with the projectivization (of one-dimensional subspaces) of the vector bundle $\mathcal O(-1)\oplus \mathcal O^{\oplus r}$ on $\CP^n_k$:
$$
Bl_P\CP^N_k\cong \CP(\mathcal E),\ \ \text{where}\ \ \mathcal E=\mathcal O_{\CP^n}(-1)\oplus \mathcal O_{\CP^n}^{\oplus r}.
$$
The proper transform $\tilde B\subset Bl_P\CP^N_k$ of $B$ is defined by  
$$
\tilde F:=x_1^{m-1}g\cdot x_0^{d-\deg (g)-m+1}y_0^m+x_1^{m-1}\sum_{j=2}^r x_0^{d-2m+1-\deg c_j}c_jy_j^m+(-1)^n x_0^{d-m-n+1}x_2\cdots x_ny_{r+1}^m ,
$$
which is
a section of the line bundle $\mathcal O_{\CP(\mathcal E)}(m)\otimes \pi^\ast \mathcal O_{\CP^n}(d-m)$ on $\CP(\mathcal E)$, where 
$$
x_0,\dots ,x_n\in H^0(\CP(\mathcal E),\pi^\ast \mathcal O_{\CP^n}(1))\cong H^0(\CP^n,\mathcal O_{\CP^n}(1))
$$
form a basis and where
$$
y_0\in H^0(\CP(\mathcal E),\mathcal O_{\CP(\mathcal E)}(1)\otimes \pi^\ast \mathcal O_{\CP^n}(-1))
$$
and 
$$
y_2,\dots ,y_{r+1}\in H^0(\CP(\mathcal E),\mathcal O_{\CP(\mathcal E)}(1))
$$ are defined
 as follows.
We have
$$
\pi_\ast \mathcal O_{\CP(\mathcal E)}(1)\cong \mathcal E^\vee=\mathcal O_{\CP^n}(1)\oplus \mathcal O_{\CP^n}^{\oplus r}
$$
and 
$$
y_0\in H^0(\CP(\mathcal E),\mathcal O_{\CP(\mathcal E)}(1)\otimes \pi^\ast \mathcal O_{\CP^n}(-1))\cong H^0(\CP^n,\mathcal O_{\CP^n})\cong k
$$
is any generator, while for $j=2,\dots ,r+1$,
$$
y_j\in H^0(\CP(\mathcal E),\mathcal O_{\CP(\mathcal E)}(1))\cong H^0(\CP^n,\mathcal O_{\CP^n}(1))\oplus H^0(\CP^n,\mathcal O_{\CP^n})^{\oplus r}
$$
is a generator of the $(j-1)$-th summand of $H^0(\CP^n,\mathcal O_{\CP^n})^{\oplus r}\cong k^{\oplus r}$. 
In the above coordinates, the exceptional divisor of $Bl_P\CP^N_k\to \CP^N_k$ is given by $y_0=0$. 

Since $F$ has degree $d$, multiplicity $d-m$ along $P$ and $m\mid d$, the divisor $\tilde B\subset Bl_P\CP^N_k$ is (uniquely) divisible by $m$ in the Picard group of $Bl_P\CP^N_k$ and so the cyclic $m:1$ cover 
$$
Y\longrightarrow Bl_P\CP^N_k,
$$
branched along $\tilde B$ 
exists.
In the above coordinates, $Y$ is given by 
$$
y_1^m=x_1^{m-1}g\cdot x_0^{d-\deg (g)-m+1}y_0^m+x_1^{m-1}\sum_{j=2}^r x_0^{d-2m+1-\deg c_j}c_jy_j^m+(-1)^n x_0^{d-m-n+1}x_2\cdots x_ny_{r+1}^m ,
$$
where $y_1$ is a new coordinate.

Let $Z\to \CP^N_k$ be the cyclic $m:1$ cover, branched along $B=\{F=0\}$.
We claim that the Stein factorization of $Y\to \CP^N_k$ induces a birational morphism
$$
\sigma:Y\longrightarrow Z,
$$ 
whose exceptional locus is given by the divisor
$$
D:=\{y_0=0\}\subset Y.
$$ 
To see this, note that the natural morphism $Y\to \CP^N_k$ contracts the divisor $D$ to $P\subset \CP^N_k$, while $Y\setminus D\to \CP^N_k\setminus P$ is finite of degree two.
Moreover, the restrictions of $Y\to \CP^N_k$ and $Z\to \CP^N_k$ to $\CP^N_k\setminus P$ are naturally isomorphic.
Let $Y\to Z'\to \CP^N_k$ be the Stein factorization of  $Y\to \CP^N_k$.
Then $Z'$ is a normal projective variety with a finite morphism $Z'\to \CP^N_k$ which coincides with the cyclic cover $Z\to \CP^N_k$ away from $P$.
In particular, there is natural birational map $Z'\dashrightarrow Z$ over $\CP^N_k$ which is an isomorphism over $\CP^n_k\setminus P$.
Serre's criterion shows that $Z$ is normal (as it is a hypersurface in a smooth variety, hence Cohen--Macauly, and regular in codimension one because $\{F=0\}$ is reduced).
Let $W\to Z'$ be a birational morphism from a normal projective variety $W$ so that $Z'\dashrightarrow Z$ induces a morphism $W\to Z$.
By Zariski's main theorem, $W\to Z'$ has connected fibres.
Since the following diagram commutes
$$
\xymatrix{
W\ar[r] \ar[d] & Z \ar[d]\\
Z'\ar[r] &\CP^N_k,
}
$$
and $Z\to \CP^N_k$ and $Z'\to \CP^N_k$ are both finite,  any fibre of $W\to Z'$ is contracted by $W\to Z$.
The rigidity lemma thus shows that $W\to Z$ factors through $W\to Z'$ and hence descends to a morphism $Z'\to Z$, which must be an isomorphism by Zariski's main theorem, because it is a birational finite morphism between normal varieties.
The composition $Y\to Z'\stackrel{\sim}\to Z$ thus induces the birational morphism $\sigma:Y\longrightarrow Z$ whose exceptional locus is given by $D=\{y_0=0\}$, as we want.

The  projection $\pi:Bl_P\CP^N_k\to \CP^n_k$ to the $x_i$-coordinates induces a morphism $f:Y\to \CP^n_k$ whose generic fibre is given by the hypersurface in $\CP^{r}_{k(\CP^n)}$, given by 
$$
y_1^m=x_1^{m-1}g y_0^m+x_1^{m-1}\sum_{j=2}^r  c_jy_j^m+(-1)^n  x_2\cdots x_ny_{r+1}^m  ,
$$
where $k(\CP^n)=k(x_1,\dots ,x_n)$.
This is smooth because $m$ is invertible in $k$.

Multiplying the equation $y_1^m=\tilde F$ of $Y$ with $x_1$ and absorbing $x_1^m$ into the $y_i$ variables whenever possible, we find after setting $x_0=y_0=1$ that $Y$ is birational to the affine hypersurface in $\A^{N+1}_k$, given by
$$
x_1y_1^m= g +\sum_{j=2}^r c_jy_j^m+(-1)^nx_1 x_2\cdots x_ny_{r+1}^m .
$$
Since $c_1=-x_1$ by (\ref{def:ci}), this is exactly the hypersurface used in the proof of Theorem \ref{thm:torsion-order-2} in Section \ref{sec:main-proofs}.
The same argument as in the proof of that result now shows that $f:Y\to \CP^n_k$ satisfies the assumptions in Theorem \ref{thm:unramified-coho}.
In particular,
$$
\gamma:= f^\ast (x_1,\dots ,x_n)\in H^n_{nr}(k(Y)/k,\mu_m^{\otimes n})
$$
is an unramified class of order $m$.

We claim that for any alteration $\tau':Y'\to Y$, the induced alteration 
$$
\tau:=\sigma\circ \tau':Y'\to Z
$$ 
has the property that for any subvariety $E\subset \tau^{-1}(Z^{\sing})$, $\tau^\ast \gamma|_E=0$.
If $E$ does not dominate $\CP^n_k$ via $f\circ \tau'$, then the vanishing in question follows directly from item (\ref{item:vanishing}) in Theorem \ref{thm:unramified-coho}.
Otherwise, since the generic fibre of $f$ is smooth, we find that 
$$
f\circ \tau'(E)\subset D\subset Y ,
$$
where we recall that $D=\{y_0=0\}$ is the exceptional locus of $\sigma:Y\to Z$.
The generic fibre $D_\eta$ of $D\to \CP^n_k$ is the  hypersurface in $\CP^{r}_{k(\CP^n)}$, given by the equation
$$
-y_1^m+x_1^{m-1}\sum_{j=2}^r  c_jy_j^m+(-1)^n  x_2\cdots x_ny_{r+1}^m=0.
$$
Multiplying this equation through by $x_1$ and absorbing $x_1^m$ into the $y_i$-variables whenever possible, we find that in these new coordinates $D_\eta\subset \CP^r_{k(\CP^n)}$ is given by
$$
-x_1y_1^m+ \sum_{j=2}^r  c_jy_j^m+(-1)^n  x_1x_2\cdots x_ny_{r+1}^m=0.
$$
Since $c_1=-x_1$, we find that $D_\eta$ is a subvariety of the hypersurface $X_{x_1,\dots ,x_n}\subset \CP^{2^n-1}_L$ over $L=k(\CP^n)$ from Corollary \ref{cor:Pfister-forms}.
Since $E$ dominates $\CP^n_k$ in the present case, the generic point of $\tau'(E)\subset D$ lies on the generic fibre $D_\eta$ of $D\to \CP^n_k$ and so
$\tau^\ast \gamma|_E=0$ follows from Corollary \ref{cor:Pfister-forms}.
Altogether we conclude that for any subvariety $E\subset \tau^{-1}(Z^{\sing})$, $\tau^\ast \gamma|_E=0$, as we want.

By what we have seen above, $Z$ satisfies the assumptions of the special fibre in the degeneration technique from
Proposition \ref{prop:degeneration}.
Hence, Proposition \ref{prop:degeneration}  implies that $m$ divides the torsion order of a cyclic $m:1$ cover of $\CP^N_k$, branched along a very general hypersurface of degree $d$, as the latter specializes to $Z$, see Section \ref{subsec:degeneration}. 
This concludes the proof.
\end{proof}

\section{Explicit examples} \label{sec:examples}

In this section we give some explicit examples of hypersurfaces with large torsion orders that are defined over small fields, such as $\Q(\mathbb A^1)$ or $\F_p(\mathbb A^2)$.
We leave it to the reader to write down similar explicit examples of cyclic covers as in Theorem \ref{thm:cyclic-cover} over fields such as $\Q(\mathbb A^1)$ or $\F_p(\mathbb A^2)$.

Fix a positive integer $m$ and let $p$ be a prime number coprime to $m$.
Let $N\geq 3$ be an integer and write $N=n+r$ with $2^{n-1}-2\leq r\leq 2^n-2$. 
Let $x_0,x_1,\dots ,x_n,y_1,\dots ,y_{r+1}$ be homogeneous coordinates on $\CP^{N+1}$ and consider the following polynomials:
\begin{align*} 
&g(x_0,\dots ,x_n):=p\cdot \left(  \sum_{i=0}^n x_i^{\lceil \frac{n+1}{m} \rceil} \right)  ^m-(-1)^{n}x_0^{m \lceil\frac{n+1}{m}\rceil -n}x_1x_2\dots x_n , \\
&F:=g(x_0,\dots ,x_n)\cdot x_0^{m+n-\deg (g)}+\sum_{j=1}^r x_0^{n-\deg c_j}c_j(x_1,\dots ,x_n)y_j^m+(-1)^n x_1x_2\cdots x_ny_{r+1}^m ,
\end{align*}
where $c_j(x_1,\dots ,x_n)\in \Z[x_1,\dots ,x_n]$ denote the coefficients of the Fermat--Pfister form (\ref{def:ci}).
(Explicitly, $c_j(x_1,\dots ,x_n)=  \prod_{i=1}^n(-x_i)^{\epsilon_i}$ where the $\epsilon_i\in \{0,1\}$ are such that $j=\sum_{i=1}^n \epsilon_i 2^{i-1}$.)

\begin{corollary}
In the above notation, let $d\geq m+n$ be an integer and let $s\in \C$ be a transcendental number.
Then the smooth complex projective hypersurface
\begin{align} \label{eq:explicit-example}
X:= \left\lbrace F\cdot x_0^{d-m-n}+s\cdot \left( \sum_{i=0}^n x_i^d+\sum_{j=1}^{r+1}y_j^d\right) =0 \right\rbrace \subset \CP^{N+1}_\C 
\end{align}
 is an explicit example of a smooth projective variety whose torsion order is divisible by $m$.
Moreover, this example is defined over the countable field $\Q(s)$.
\end{corollary}

\begin{proof} 
One easily checks that $X$ is smooth (e.g.\ because it degenerates to the smooth hypersurface $\left\lbrace \sum_{i=0}^n x_i^d+\sum_{j=1}^{r+1}y_j^d=0 \right\rbrace $ via $s\to \infty$, c.f.\ Section \ref{subsec:degeneration}).
Moreover, $X$ specializes via $s\to 0$ to the union $Z\cup \{x_0^{d-m-n}=0\}$, where $Z=\{F=0\}$ is the hypersurface used in the proof of Theorem \ref{thm:torsion-order-2} (where $t$ is taken to be the prime number $p$ coprime to $m$, which works in characteristic zero, as indicated in the proof of Theorem \ref{thm:torsion-order-2}).
As we have seen in the proof of Theorem \ref{thm:torsion-order-2}, the assumptions of Proposition \ref{prop:degeneration} are satisfied in this situation and so we conclude from that result that the torsion order of $X$ is divisible by $m$, as we want. 
\end{proof}

\begin{remark}
Let $p$ be a prime number coprime to $m$.
It is then easy to adjust the above arguments to obtain explicit examples of smooth projective hypersurfaces $X\subset \CP^{N+1}_k$ in the above degree range over the function field $k=\F_p(s,t)$ of $\A^2_{\F_p}$ such that the torsion order of $X_{\overline k}$ is divisible by $m$.
Explicitly, the corresponding examples will be given by replacing  the prime number $p$ in (\ref{eq:explicit-example}) by the variable $t$.
\end{remark}

\section*{Acknowledgements}  
I am grateful to Burt Totaro for useful conversations and for arranging my visit to UCLA in fall 2019, during which this work was done. 
I am also grateful to Jan Hennig for comments and to the excellent referees for many valuable comments, suggestions and corrections.
Further thanks go to a referee and Jean-Louis Colliot-Th\'el\`ene for pointing out Lemma \ref{lem:spec-prop}.


\end{document}